\documentclass[]{article}

\usepackage{amsmath} 
\usepackage{siunitx} 
\usepackage{enumerate} 
\usepackage{tikz,graphicx}
\usepackage{amssymb}
\usepackage{amsthm}
\usepackage{mathrsfs}
\usepackage{hyperref}
\usepackage{xpatch}
\usepackage{cite}
\usepackage{abstract}
\usepackage[left=3cm,right=3cm,lines=45,top=0.8in,bottom=1.0in]{geometry}
\usepackage{setspace}
\usepackage{titlesec}
\usepackage{authblk}
\usepackage{changes}
\definechangesauthor[name={Zhizhuo}, color=red]{z}
\allowdisplaybreaks[4]

\newtheorem{defn}{Definition}[section]
\newtheorem{thm}[defn]{Theorem}
\newtheorem{lem}[defn]{Lemma}

\newtheorem{prp}[defn]{Proposition}
\newtheorem*{prm}{Problem}
\newtheorem*{hyp}{Hypothesis}
\newtheorem{rem}{Remark}
\numberwithin{equation}{section}
\titleformat{\section}{\LARGE\bfseries}{\thesection}{1em}{}
\titleformat{\subsection}{\Large\bfseries}{\thesubsection}{1em}{}
\renewcommand\thesection{\arabic{section}} 
\renewcommand\thesubsection{\arabic{section}.\arabic{subsection}}

\newcommand{\keywords}[1]{\textbf{Keywords}: #1}

\begin{document}

\title{Variational inequalities of multilayer elastic systems with interlayer friction: existence and uniqueness of solution and convergence of numerical solution}

\author[1]{ Zhizhuo Zhang \thanks{zhizhuo\_zhang@163.com}}
\author[1]{ Xiaobing Nie \thanks{xbnie@seu.edu.cn}}
\author[1]{ Jinde Cao \thanks{ Corresponding author: jdcao@seu.edu.cn}}

\affil[1]{School of Mathematics, Southeast University, Nanjing, China}
\date{}

\maketitle

\begin{abstract}
Based on the mathematical-physical model of pavement mechanics, a multilayer elastic system with interlayer friction conditions is constructed. Given the complex boundary conditions, the corresponding variational inequalities of the partial differential equations are derived, so that the problem can be analyzed under the variational framework. First, the existence and uniqueness of the solution of the variational inequality is proved; then the approximation error of the numerical solution based on the finite element method is analyzed, and when the finite element space satisfies certain approximation conditions, the convergence of the numerical solution is proved; finally, in the trivial finite element space, the convergence order of the numerical solution is derived. 
The above conclusions provide basic theoretical support for solving the displacement-strain problem of multilayer elastic systems under the framework of variational inequalities.
\end{abstract}

\section{Introduction}
The pavement mechanics model has always been an extremely critical basic research topic in the field of road traffic, which is of great significance to the damage, service life, inspection and maintenance of the pavement\cite{white2001coupling,cebon1993interaction, brownjohn2007structural}.
In order to study this problem, Boussinesq first proposed the subject of elastic half-spaces\cite{boussinesq1885application}, and after continuous development, methods such as potential functions\cite{burmister1945general} and rheology theory\cite{monismith1966rheologic} were proposed to analyze the stress-displacement problem of asphalt road.
With the development of computers, the numerical solution of the analytical expressions of elastic multilayer systems has gradually become a research hotspot, in which the continuous, smooth or partial friction between layers is also gradually being studied \cite{de1979computer,hayhoe2002leaf,huang2004pavement,khazanovich2007mnlayer}.
Among them, the main tools in the process of numerical solution of analytical expressions are Hankel integral transform\cite{muki1956three}, Laplace-Hankle integral transform, Fourier transform\cite{siddharthan1998pavement}, wavelet transform and convolution integral method\cite{zhao2014viscoelastic}, etc.
However, with the increasing complexity of the load, boundary conditions and constitutive relations, the derivation of analytical solutions for multilayer elastic systems becomes extremely difficult. Therefore, the method of directly simulating complex mathematical models has gradually become the mainstream of research in recent years. Naturally, finite element software such as Abaqus, Ansys, Adina and Comsol became the main tools for pavement modeling and numerical simulation\cite{wang2011analysis,wollny2016numerical,ma2022dynamic}.

In the traditional pavement design theory, it is generally assumed that the contact state between the structural layers is completely continuous, however, the actual pavement experiments show that the interlayer structure does not satisfy this assumption, but is in a semi-bonded state or even slip state under extreme conditions\cite{hu2011effects}. The huge difference between the experiment and the ideal model directly affects the evaluation of pavement performance, and it is necessary to reduce the error by constructing a new mathematical-physical model. 
The friction coefficient model based on Coulomb's law is an important improvement to the multilayer elastic system, in which the friction coefficient is used to characterize the interlayer bonding state, and a finite element model is established to simulate the pavement structure\cite{wu2017effects,kim2011numerical,guo2016assessing,you2020assessing}.

It can be found that the numerical solution of the stress-displacement problem for the multilayer elastic system is a mathematical problem with practical application background, and it can also be abstracted as the contact problem of elastic mechanics\cite{zienkiewicz2005finite}. 
In fact, the contact problem, which describes the physical situation of a deformable elastic body in contact with other objects, has received extensive attention due to its broad physical background\cite{gladwell1980contact}. Due to the complexity of the contact problem model, it is extremely challenging to directly use the finite element method to solve the exact numerical solution. Considering that the establishment of the initial partial differential equations of elastic systems is based on the small deformation assumption and the principle of minimum potential energy\cite{marsden1994mathematical,ciarlet2021mathematical}, the study of the contact problem under the framework of variational inequalities naturally becomes the 
pivotal method\cite{kikuchi1988contact}. Under the rich theoretical support of convex analysis, optimization and variational methods,  contact problems under various types of contact conditions and constitutive equations are systematically studied. Han, Weimin, et al. studied the quasi-static contact problems of viscoelastic and viscoplastic materials\cite{han2002quasistatic}. 
Then, Eck Christoph et al. studied the contact between two elastic bodies without friction\cite{eck2003convergence}.
In order to further simulate the real contact problem, Bayada Guy studied the contact problem of two elastic bodies with Coulomb friction\cite{bayada2008convergence}. 
And Krause Rolf gave a non-smooth multi-scale method to solve the numerical solution of the frictional contact problem of two bodys\cite{krause2009nonsmooth}.

However, the study of the contact problem under the framework of variational inequalities has never been applied to the modeling of pavement mechanics. In our study, a system of partial differential equations for a multilayer elastic system with interlayer Coulomb's friction will be constructed under the assumption of small deformation and Saint-Venant's principle. Among them, the constitutive equation will be described by a nonlinear operator satisfying a certain regularity condition, and the interlayer friction condition will be described by a functional related to the interlayer normal stress, which ensures the wide applicability of nonlinear multilayer elastic systems. At the same time, in order to be more in line with the real road conditions, the foundation is assumed to be a deformable body, that is, the contact condition between the elastic body and the foundation has normal compliance.
Then, the variational inequality derived from this system of partial differential equations will be introduced, and based on Banach's fixed point theorem \cite{han2000evolutionary} and the theory of elliptic variational inequality \cite{capatina2014variational}, the existence and uniqueness of the solution of this inequality is proved. 
Based on the finite element method, the approximation error between the numerical solution and the true solution and its convergence properties are also studied. 
The existence and uniqueness of the solution theoretically supports the rationality of studying the pavement displacement-stress relationship under the framework of variational inequalities. And the convergence of the numerical solutions obtained by the finite element method proves the feasibility of solving the corresponding numerical solutions under the framework of variational inequalities. Therefore, it is a feasible method to study the mechanical response of pavement structure under the variational inequality method.

The remainder of the paper will be organized as follows. In the section 2, the mathematical-physical model of the multilayer elastic system will be introduced, and then its corresponding variational inequalities will be derived. In section 3, the existence and uniqueness of the solution of the variational inequality will be proved. In section 4, discrete approximation errors for numerical solutions in abstract finite element spaces will be first derived. Then, under certain spatial assumptions, the convergence of the numerical solution will be further proved. Finally, in the most trivial finite element space, the order of convergence will be derived. In section 5, the entire paper will be summarized.

\section{The physical model and variational formulation}

\subsection{The physical model and PDEs}

\begin{figure}[t]
    \centering
    \includegraphics[width=6in]{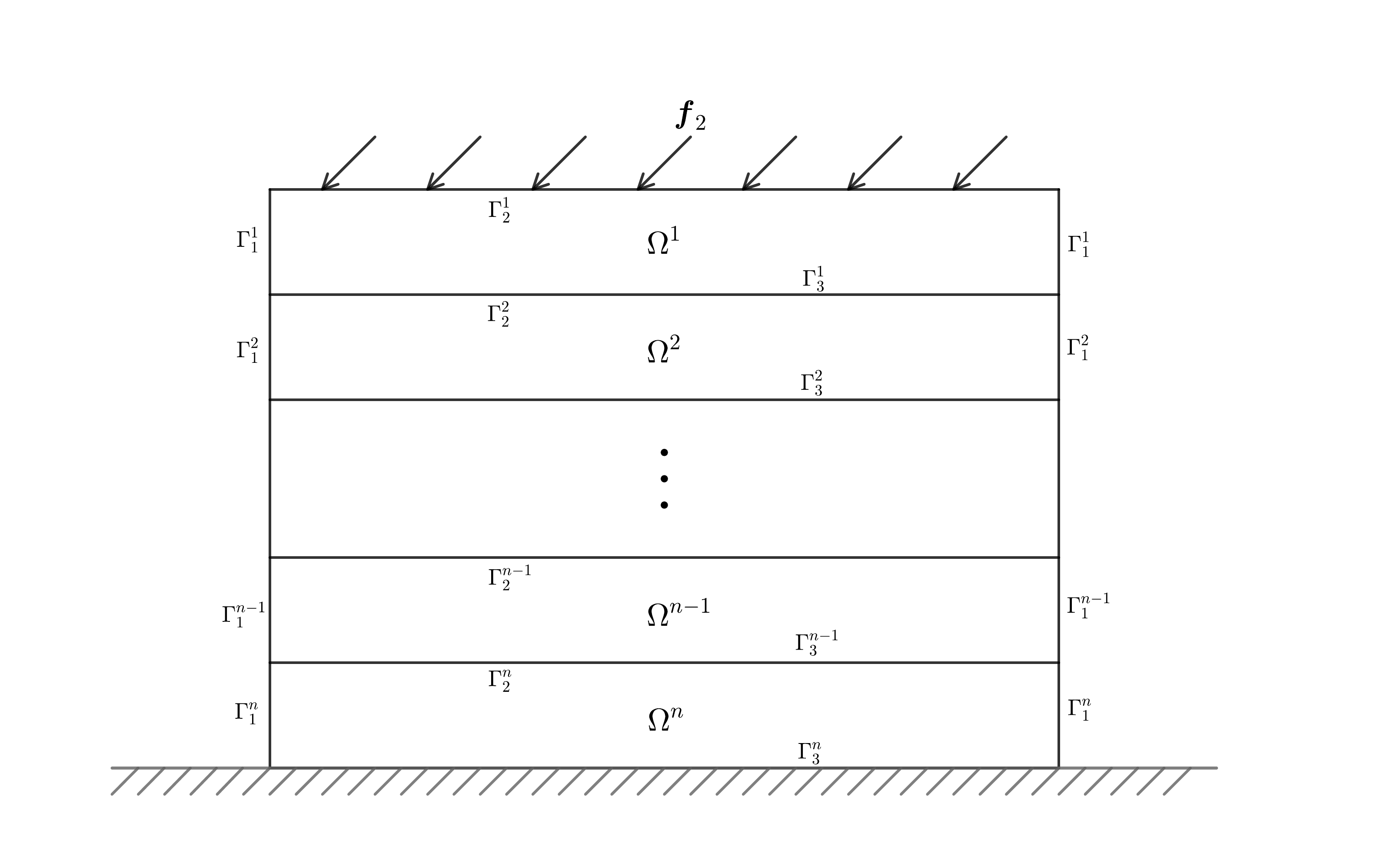}
    \caption{The physical model of multilayer contact problem}
    \label{fig1}
\end{figure}

In this section, a multilayer elastic system based on contact conditions with friction will be introduced. The physical model of the contact problem of the multilayer elastic system considered in this study is shown in Fig 1. The open bounded connected region occupied by each layer of elastic body is denoted as $\Omega^{i}\in\mathbb{R}^{d}$, where $d=2$ or $3$ and $i=1,2,\ldots,n$. Therefore, the total domain occupied by the multilayer elastic system is denoted as $\Omega=\cup_{i=1}^{n}\Omega^{i}$. Then, the boundary $\Gamma^{i}=\partial \Omega^{i}$ is assumed to be Lipschitz continuous and satisfy the decomposition $\Gamma^{i}=\cup_{j=1}^{3} \bar{\Gamma}_{j}^{i}$, where 
$\Gamma_{1}^{i}$, $\Gamma_{2}^{i}$, $\Gamma_{3}^{i}$ are mutually disjoint relatively open sets. 
Similarly, let $\Gamma_{j}=\cup_{i=1}^{n}\bar{\Gamma}_{j}^{i}$, $j=1,2,3$.
Let $\operatorname{meas}_{d}\left(\cdot\right)$ be the d-dimensional Lebesgue-measure of the set, and suppose $\operatorname{meas}_{d-1}\left(\Gamma_{1}^{i}\right)>0$ for $i=1,2,\ldots,n$. For $x\in\Gamma_{j}^{i}$, the unit outward normal vector of $x$ is denoted by $\nu$. Note that in the multilayer elastic system, when no force is applied, the boundaries $\Gamma_{2}^{i+1}$ and $\Gamma_{3}^{i}$ coincide with each other. 
However, these two boundaries are not equivalent to each other, because they belong to different elastic bodies and their unit outward normal vectors are in opposite directions.
For the convenience of calculation, the unit outward normal vectors of the boundary $\Gamma_{2}^{i}$ and $\Gamma_{3}^{i}$ ($i=1,2,\ldots,n$) are denoted as $\boldsymbol{\alpha}^{i}$ and $\boldsymbol{\beta}^{i}$, respectively, and let the friction boundary $\Gamma_{c}^{i}=\Gamma_{3}^{i}$.

Then, the external forces and contact conditions of the multilayer elastic system will be defined. First, the body force on the elastic body of the $i$-th layer is denoted as $\boldsymbol{f}_{0}^{i}$, $i=1,2,\ldots,n$. Then, the elastic body of the $i$-th layer is fixed on the boundary $\Gamma_{1}^{i}$, that is, the Dirichlet boundary condition is satisfied on the boundary $\Gamma_{1}^{i}$, so the displacement field vanishes there. The surface traction force $\boldsymbol{f}_{2}$ acts on the boundary $\Gamma_{2}^{1}$ of the first layer of elastic body, and for the $i$-th layer of elastic body ($i=2,3,\ldots,n$), its boundary $\Gamma_{2}^{i}$ is acted by the force from the upper layer of elastic body, which will be subject to the Coulomb's friction condition. Finally, on the boundary $\Gamma_{3}^{i}$ ($i=1,2,\ldots,n-1$), the body is in Coulomb's frictional contact with the next layer of body, while on the boundary $\Gamma_{3}^{n}$, the frictional contact with normal compliance is satisfied, which is based on foundation assumptions.
\begin{rem}
It is worth noting that the Saint-Venant principle ensures that the boundary condition can simulate the real road conditions well, and the quasi-static condition ensures that the small deformation can reach equilibrium in the system in a very short time without considering the effect of acceleration. Therefore, under these two assumptions, the model can well simulate the instantaneous mechanical response of asphalt pavement.
\end{rem}

Under the action of force, the elastic body of the $i$-th layer will deform, and its displacement field is denoted as $\boldsymbol{u}^{i}: \Omega^{i} \to \mathbb{R}^{d}$ ($i=1,2,\ldots,n$). At the same time, the stress tensor inside the object is denoted as $\boldsymbol{\sigma}^{i}: \Omega^{i} \to \mathbb{S}^{d}$, where $\mathbb{S}^{d}$ represents the space of second order symmetric tensors on $\mathbb{R}^{d}$. 
Then, "$\cdot$" and "$:$" represent the inner product in $\mathbb{R}^{d}$ and $\mathbb{S}^{d}$, respectively, and "$|\cdot|$" represents the Euclidean norm. And the space of displacement field and stress function is defined as:
$$
\begin{aligned}
V^{i} &=\left\{\boldsymbol{v}=\left(v_{k}\right) \in\left(H^{1}(\Omega^{i})\right)^{d}: \boldsymbol{v}=\mathbf{0} \text { on } \Gamma^{i}_{1}\right\}, \\
Q^{i} &=\left\{\boldsymbol{\tau}=\left(\tau_{kl}\right) \in\left(L^{2}(\Omega^{i})\right)^{d \times d}: \tau_{lk}=\tau_{lk}, 1 \leqslant k, l \leqslant d\right\}, \\
Q^{i}_{1} &=\left\{\boldsymbol{\tau} \in Q^{i}: \operatorname{Div} \boldsymbol{\tau} \in\left(L^{2}(\Omega^{i})\right)^{d}\right\},
\end{aligned}
$$
where $i=1,2\ldots,n$ and $H^{1}(\cdot)=W^{1,2}(\cdot)$ is Sobolev space. Therefore, on the domain $\Omega$, let $\boldsymbol{u}=\left(\boldsymbol{u}^{1}, \boldsymbol{u}^{2}, \ldots , \boldsymbol{u}^{n} \right)$ and $\boldsymbol{\sigma}=\left(\boldsymbol{\sigma}^{1}, \boldsymbol{\sigma}^{2}, \ldots , \boldsymbol{\sigma}^{n} \right)$. Then $\boldsymbol{u}\in V$ and $\boldsymbol{\sigma}\in Q_{1}$, where
$$
V = V^{1} \times V^{2} \times \cdots \times V^{n} \text{ and } 
Q_{1}= Q_{1}^{1} \times Q_{1}^{2} \times \cdots \times Q_{1}^{n}.
$$
Based on the definition, it is easy to verify that the above spaces are Hilbert spaces, so the standard norm can be defined as follows:
$$
(\boldsymbol{u}^{i},\boldsymbol{v}^{i})_{H^{1}\left(\Omega^{i}\right)^{d}} = \sum_{k=1}^{d} (u^{i}_{k},v^{i}_{k})_{H^{1}\left(\Omega^{i}\right)}
$$
and
$$
(\boldsymbol{\sigma}^{i},\boldsymbol{\tau}^{i})_{L^{2}\left(\Omega^{i}\right)^{d}} = \sum_{k,l=1}^{d} (\sigma^{i}_{kl},\tau^{i}_{kl})_{L^{2}\left(\Omega^{i}\right)},
$$
where $\boldsymbol{u}^{i},\boldsymbol{v}^{i}\in V^{i}$ and $\boldsymbol{\sigma}^{i},\boldsymbol{\tau}^{i}\in Q^{i}$ ($i=1,2,\ldots,n$). From this, the standard inner product on spaces $V$ and $Q$ can be defined as:
$$
(\boldsymbol{u},\boldsymbol{v})_{H^{1}} = \sum_{i=1}^{n} \sum_{k=1}^{d} (u^{i}_{k},v^{i}_{k})_{H^{1}\left(\Omega^{i}\right)}
$$
and
$$
(\boldsymbol{\sigma},\boldsymbol{\tau})_{L^{2}} = \sum_{i=1}^{n} \sum_{k,l=1}^{d} (\sigma^{i}_{kl},\tau^{i}_{kl})_{L^{2}\left(\Omega^{i}\right)}.
$$
Based on the assumption of small deformation, the stress tensor $\boldsymbol{\varepsilon}\left(\boldsymbol{v}^{i}\right)\in Q^{i}$ can be defined as:
$$
\boldsymbol{\varepsilon}\left(\boldsymbol{v}^{i}\right)=\frac{1}{2}\left(\nabla \boldsymbol{v}^{i}+\left(\nabla \boldsymbol{v}^{i}\right)^{T}\right).
$$
Since $\operatorname{meas}_{d-1}\left(\Gamma_{1}^{i}\right)>0$, Korn's inequality holds\cite{kikuchi1988contact}:
$$
\|\boldsymbol{v}^{i}\|_{H^{1}(\Omega^{i})^{d}} \leqslant c_{k}^{i}\|\boldsymbol{\varepsilon}(\boldsymbol{v^{i}})\|_{L^{2}\left(\Omega^{i}\right)^{d}} \quad \forall \boldsymbol{v}^{i} \in V^{i},
$$
where $c_{k}^{i}>0$ is a constant depending only on $\Omega^{i}$ and $\Gamma_{1}^{i}$. Based on Korn's inequality, the inner product can be defined:
$$
(\boldsymbol{u}^{i}, \boldsymbol{v}^{i})_{V^{i}} =(\boldsymbol{\varepsilon}(\boldsymbol{u}^{i}), \boldsymbol{\varepsilon}(\boldsymbol{v}^{i}))_{L^{2}\left(\Omega^{i}\right)^{d}} \quad \forall \boldsymbol{u}^{i}, \boldsymbol{v}^{i} \in V^{i},
$$
and the norm $\|\cdot\|_{V^{i}}$ induced by this inner product is equivalent to the standard norm $\|\cdot\|_{H^{1}(\Omega^{i})^{d}}$ in $V^{i}$ space. Therefore, $\left(V^{i}, \|\cdot\|_{V^{i}}\right)$ is a real Hilbert space.
Similarly, the inner product can be defined in space $V$:
$$
(\boldsymbol{u}, \boldsymbol{v})_{V} = \sum_{i=1}^{n} (\boldsymbol{\varepsilon}(\boldsymbol{u}^{i}), \boldsymbol{\varepsilon}(\boldsymbol{v}^{i}))_{L^{2}\left(\Omega^{i}\right)^{d}} \quad \forall \boldsymbol{u}, \boldsymbol{v} \in V,
$$
It is easy to verify that $\left(V, \|\cdot\|_{V}\right)$ is also a real Hilbert space. Furthermore, according to the trace theorem on Sobolev spaces\cite{adams2003sobolev}, there exists a constant $c_{t}^{i}$ that depends only on $\Omega^{i}$ and $\Gamma^{i}$, such that the following inequality holds:
\begin{equation}\label{2.0}
\|\boldsymbol{v}^{i}\|_{L^{2}\left(\Gamma_{3}^{i}\right)^{d}} \leqslant c_{t}^{i}\|\boldsymbol{v}^{i}\|_{V^{i}},~ \forall \boldsymbol{v}^{i} \in V^{i}.
\end{equation}
In addition, we also assume that for the region $e^{i}$ satisfying $\operatorname{meas}_{d-1}\left(e^{i}\right)>0$ ($i=1,2,\ldots,n$) on the boundary $\partial\Omega^{i}$, the following inverse estimates hold:
\begin{align}\label{2.01}
\|\boldsymbol{\varepsilon}(\boldsymbol{v}^{i})\|_{L^{2}\left(e^{i}\right)^{d}}
\leqslant c_{h}^{i} \|\boldsymbol{\varepsilon}(\boldsymbol{v}^{i})\|_{L^{2}\left(\Omega^{i}\right)^{d}},~ \forall \boldsymbol{v}^{i} \in V^{i},
\end{align}
where the constant $c_{h}^{i}>0$ is determined by $e^{i}$ and $\Omega^{i}$. 
This inequality was originally proposed for the piecewise polynomial function space \cite{schwab1998p,riviere2003discontinuous}. In the Sobolev space $H_{1}(\Omega)^{d}$, this inequality can be proved by the trace theorem of Besov space, Korn's inequality and Soboblev's inequality.

In order to facilitate the definition of contact boundary conditions, the displacement and stress on the boundaries $\Gamma^{i}_j$ ($j=2,3$) are decomposed as follows:
\begin{align*}
&v_{\beta}^{i}=\boldsymbol{v}^{i} \cdot \boldsymbol{\beta}^{i}, && \boldsymbol{v}_{\eta}^{i}=\boldsymbol{v}^{i}-v_{\beta}^{i}\cdot\boldsymbol{\beta}^{i}, && \boldsymbol{v}^{i}\in \Gamma^{i}_3;\\
&v_{\alpha}^{i}=\boldsymbol{v}^{i} \cdot \boldsymbol{\alpha}^{i}, && \boldsymbol{v}_{\tau}^{i}=\boldsymbol{v}^{i}-v_{\alpha}^{i}\cdot\boldsymbol{\alpha}^{i}, && \boldsymbol{v}^{i}\in \Gamma^{i}_2;\\
&\sigma_{\beta}^{i}=\boldsymbol{\beta}^{i}\cdot\boldsymbol{\sigma}^{i} \cdot \boldsymbol{\beta}^{i}, && \boldsymbol{\sigma}_{\eta}^{i}=\boldsymbol{\sigma}^{i} \cdot \boldsymbol{\beta}^{i}-\sigma_{\beta}^{i}\cdot\boldsymbol{\beta}^{i}, && \boldsymbol{\sigma}^{i}\in Q^{i};\\
&\sigma_{\alpha}^{i}=\boldsymbol{\alpha}^{i}\cdot\boldsymbol{\sigma}^{i} \cdot \boldsymbol{\alpha}^{i}, && \boldsymbol{\sigma}_{\tau}^{i}=\boldsymbol{\sigma}^{i} \cdot \boldsymbol{\alpha}^{i}-\sigma_{\alpha}^{i}\cdot\boldsymbol{\alpha}^{i}, && \boldsymbol{\sigma}^{i}\in Q^{i}.
\end{align*}
On the boundary $\Gamma_{3}^{i}$, $v_{\beta}^{i}$ and $\boldsymbol{v}_{\eta}^{i}$ represent the normal and tangential displacements, respectively, and $\sigma_{\beta}^{i}$ and $\boldsymbol{\sigma}_{\eta}^{i}$ represent the normal and tangential stresses, respectively. Similarly, on the boundary $\Gamma_{2}^{i}$, $v_{\alpha}^{i}$, $\boldsymbol{v}_{\tau}^{i}$, $\sigma_{\alpha}^{i}$ and $\boldsymbol{\sigma}_{\tau}^{i}$ represent the displacements and stresses on the corresponding components.
Based on the boundary $\Gamma_{c}^{i}=\Gamma_{3}^{i}$, let $v_{N}^{i}=v_{\beta}^{i}$, $\boldsymbol{v}_{T}^{i}=\boldsymbol{v}_{\eta}^{i}$, $\sigma_{N}^{i}=\sigma_{\beta}^{i}$ and $\boldsymbol{\sigma}_{T}^{i}=\boldsymbol{\sigma}_{\eta}^{i}$.
In order to characterize the discontinuity of the displacement field on the friction boundary, the jump operator $[\cdot]$ is defined as:
\begin{align*}
    \left[{v}_{N}^{i}\right]&=\boldsymbol{v}^{i} \cdot \boldsymbol{\beta}^{i}+\boldsymbol{v}^{i+1} \cdot \boldsymbol{\alpha}^{i+1}\\
    \left[\boldsymbol{v}_{T}^{i}\right]&=\boldsymbol{v}_{\eta}^{i} - \boldsymbol{v}_{\tau}^{i+1},
\end{align*}
where $\boldsymbol{v}^{i}\in \Gamma^{i}_{3}$ and $\boldsymbol{v}^{i+1}\in \Gamma^{i+1}_{2}$ ($i=1,2,\ldots,n-1$).
Note that the contact between the elastic bodies on the contact boundary $\Gamma_{c}^{i}$ ($i=1,2,\ldots n-1$) satisfies the non-penetration condition, that is $\left[{v}_{N}^{i}\right]\leqslant 0$.

Under the above conditions, the frictional contact problem of the multilayer elastic system can be described by the following partial differential equation:
\begin{prm}[$P_{0}$]\label{p0}
Find a displacement field $\boldsymbol{u}^{i}:\Omega^{i} \rightarrow \mathbb{R}^{d}$ and the stress field $\boldsymbol{\sigma}^{i}: \Omega^{i} \rightarrow \mathbb{S}^{d}$ ($i=1,2,\ldots,n$) such that:
\begin{align}
&\boldsymbol{\sigma}^{i}=\mathcal{A}^{i} \boldsymbol{\varepsilon}(\boldsymbol{u}^{i}) && \text { in } \Omega^{i},\label{2.1}\\
&\operatorname{Div} \boldsymbol{\sigma}^{i}+\boldsymbol{f}_{0}^{i}=\mathbf{0} && \text { in } \Omega^{i},\label{2.2}\\
&\boldsymbol{u}^{i}=\mathbf{0} && \text { on } \Gamma_{1}^{i},\label{2.3}\\
&\boldsymbol{\sigma}^{1} \cdot \boldsymbol{\alpha}^{1}=\boldsymbol{f}_{2} && \text { on } \Gamma_{2}^{1},\label{2.4}\\
&\left.
\begin{aligned}
& \sigma^{n}_{\beta} = -g^{n}_{N}(u^{n}_{\beta}),~ |\boldsymbol{\sigma}_{\eta}^{n}|\leqslant g^{n}_{T}(u^{n}_{\beta})\\
& |\boldsymbol{\sigma}_{\eta}^{n}| < g^{n}_{T}(u^{n}_{\beta}) \Rightarrow
\boldsymbol{u}_{\eta}^{n} = 0\\
& |\boldsymbol{\sigma}_{\eta}^{n}| = g^{n}_{T}(u^{n}_{\beta}) \Rightarrow
\boldsymbol{\sigma}_{\eta}^{n} = -\lambda\boldsymbol{u}_{\eta}^{n}, ~\lambda\geqslant 0
\end{aligned}
\right\} &&  \text { on } \Gamma_{3}^{n},\label{2.5}\\
&\sigma^{i}_{\alpha} = -\sigma^{i-1}_{\beta},~
\boldsymbol{\sigma}_{\tau}^{i} = \boldsymbol{\sigma}_{\eta}^{i-1},
&&\text { on } \Gamma_{2}^{i}, ~i\ne 1, \label{2.6}\\
&\left.
\begin{aligned}
& \left[{u}_{N}^{i}\right] \leqslant 0,~ \sigma^{i}_{N}\cdot[u_{N}^{i}] = 0,~ |\boldsymbol{\sigma}_{T}^{i}|\leqslant g^{i}_{T}(\sigma^{i}_{N})\\
& |\boldsymbol{\sigma}_{T}^{i}| < g^{i}_{T}(\sigma^{i}_{N}) \Rightarrow
[\boldsymbol{u}_{T}^{i}] = 0\\
& |\boldsymbol{\sigma}_{T}^{i}| = g^{i}_{T}(\sigma^{i}_{N}) \Rightarrow
\boldsymbol{\sigma}_{T}^{i} = -\lambda[\boldsymbol{u}_{T}^{i}], ~\lambda\geqslant 0
\end{aligned}
\right\} &&  \text { on } \Gamma_{c}^{i},~i\ne n.\label{2.7}
\end{align}
\end{prm}
In the above partial differential equation, formula (\ref{2.1}) represents the constitutive relation of stress and strain, where $\mathcal{A}^{i}$ is a given nonlinear operator, which is called elastic operator. Equation (\ref{2.2}) represents the equilibrium equation between stress and body force; formulations (\ref{2.3}) and (\ref{2.4}) are displacement field constraints (Dirichlet boundary conditions) and stress boundary conditions, respectively.

Then, formula (\ref{2.5}) represents the frictional contact condition between the elastic layer and the foundation on the boundary $\Gamma_{3}^{n}$. Based on the physical background of the pavement structure, the foundation usually does not have a rigid structure, so the contact condition of the contact surface $\Gamma_{3}^{n}$ between the elastic layer and the foundation has normal compliance. There are many ways to define the contact function $g_{N}^{n}\left(\cdot\right)$, such as\cite{han2002quasistatic}:
\begin{equation}\label{2.8}
g_{N}^{n}\left(r\right)=c^{n}\left(r_{+}\right)^{m} ~~\text{or}~~
g_{N}^{n}\left(r\right)=\left\{\begin{array}{lll}
c^{n} r+ & \text { if } & r \leqslant r_{0}^{n} \\
c^{n} r_{0} & \text { if } & r>r_{0}^{n}
\end{array}\right.,
\end{equation}
where $r_{+}=\max \{0, r\}$ and $r_{0}^{n}$ is a positive coefficient related to the wear and hardness of the surface. Then, based on Coulomb's friction law, there is an upper bound function $g_{T}^{n}\left(\cdot\right)$ for the tangential stress $\boldsymbol{\sigma}_{\eta}^{n}$ representing the maximum frictional resistance, which is related to the normal stress $\sigma_{\beta}^{n}$.
When the tangential stress is strictly less than the upper bound $g_{T}^{n}\left(\cdot\right)$, the elastic body is in a state of static friction, that is, no displacement occurs on the boundary $\Gamma_{3}^{n}$, which is so-called the stick state. 
But once the tangential stress reaches the upper bound, that is, the equation is established, the elastic body is displaced on the boundary $\Gamma_{3}^{n}$, which is called the slip state. 
Note that based on the previous assumptions, the normal stress $\sigma_{\beta}^{n}$ is determined by the normal displacement $u_{\beta}^{n}$, so it may be possible to make the friction $g_{T}^{n}\left(\cdot\right)$ also a function of $u_{\beta}^{n}$. For example, under the classical Coulomb's friction law, we have
\begin{equation}\label{2.9}
    g_{T}^{n}\left(u_{\beta}^{n}\right)=\mu^{n} g_{N}^{n}\left(u_{\beta}^{n}\right)
\end{equation}
with $\mu^{n} \geqslant 0$. Based on the assumptions of thermodynamics, the modified Coulomb friction law has the following relationship:
\begin{equation}\label{2.10}
g_{T}^{n}\left(u_{\beta}^{n}\right)=\mu^{n} g_{N}^{n}\left(u_{\beta}^{n}\right) \cdot \left(1-\delta^{n} g_{N}^{n}\left(u_{\beta}^{n}\right)\right)_{+},
\end{equation}
where $\delta^{n}$ is a small positive coefficient related to the wear and hardness of the surface. The physical meaning of the above friction law is that if the normal stress exceeds the upper bound $1/\delta^{n}$, the contact surface will be destroyed and the tangential resistance will disappear.

Finally, formulations (\ref{2.6}) and (\ref{2.7}) represent the frictional contact conditions between the layers of elastic bodies. Here, Equation (\ref{2.6}) indicates that the force between the two elastic bodies is equal in magnitude and opposite in direction. Since both sides of the contact surface $\Gamma_{c}^{i}$ are elastic bodies that deform under the action of force, relative displacements, ie, $\left[u_{N}^{i}\right]$ and $\left[\boldsymbol{u}_{T}^{i}\right]$, must be used to describe the law of friction. Based on the relationship between friction force and displacement, formula (\ref{2.7}) can also be easily deduced, in which the friction function is defined in a similar way to (\ref{2.9})-(\ref{2.10}), and its coefficients  $\mu^{i}$ and $\delta^{i}$ ($i=1,2,\ldots,n-1$) are determined by the materials of the two objects in contact. Furthermore, it is easy to prove that 
formulation (\ref{2.7}) is equivalent to:
\begin{equation}\label{2.11}
\left\{\begin{aligned}
&\sigma_{N}^{i} \cdot\left[u_{N}^{i}\right]=0,~~\left|\boldsymbol{\sigma}_{T}^{i}\right| \leqslant g_{T}^{i}\left(\sigma_{N}^{i}\right) \\
&g_{T}^{i}\left(\sigma_{N}^{i}\right)\left|\left[\boldsymbol{u}_{T}^{i}\right]\right|+\boldsymbol{\sigma}_{T}^{i} \cdot \left[\boldsymbol{u}_{T}^{i}\right]=0
\end{aligned}\right. .
\end{equation}

\subsection{Variational inequality}
In order to derive the variational inequality of problem $P_{0}$, the space $\mathcal{K}$ needs to be defined as follows:
$$
\mathcal{K}=\left\{\boldsymbol{v} \in V : ~\left[v_{N}^{i}\right] \leqslant 0 \text {, a.e. on } \Gamma_{c}^{i},~ i=1,2,\ldots,n-1\right\} \subset V.
$$
Obviously, the the space $\mathcal{K}$ is the closed subspace of $V$.
Then, the elasticity operators $\mathcal{A}^{i}$ in problem $P_{0}$ are assumed to satisfy the following properties:
\begin{equation}\label{2.12}
\left\{\begin{aligned}
&\text { (a) } \mathcal{A}^{i}: \Omega^{i} \times \mathbb{S}^{d} \rightarrow \mathbb{S}^{d}; \\ 
&\text { (b) There exists } L_{A}^{i}>0 \text { such that } \\ 
&~~~~~~~~\left|\mathcal{A}^{i}\left(\boldsymbol{x}, \boldsymbol{\varepsilon}_{1}\right)-\mathcal{A}^{i}\left(\boldsymbol{x}, \boldsymbol{\varepsilon}_{2}\right)\right| \leqslant L_{A}^{i}\left|\boldsymbol{\varepsilon}_{1} - \boldsymbol{\varepsilon}_{2}\right|,  \\
&~~~~~~~~\forall \boldsymbol{\varepsilon}_{1}, \boldsymbol{\varepsilon}_{2} \in \mathbb{S}^{d} \text{ a.e. } \boldsymbol{x} \in \Omega^{i};  \\ 
&\text { (c) There exists } M^{i}>0 \text { such that } \\ 
&~~~~~~~~\left(\mathcal{A}^{i}\left(\boldsymbol{x}, \boldsymbol{\varepsilon}_{1}\right)-\mathcal{A}^{i}\left(\boldsymbol{x}, \boldsymbol{\varepsilon}_{2}\right)\right) : \left(\boldsymbol{\varepsilon}_{1} - \boldsymbol{\varepsilon}_{2}\right) \geqslant M^{i}\left|\boldsymbol{\varepsilon}_{1}-\boldsymbol{\varepsilon}_{2}\right|^{2},  \\
&~~~~~~~~\forall \boldsymbol{\varepsilon}_{1}, \boldsymbol{\varepsilon}_{2} \in \mathbb{S}^{d} \text { a.e. } \boldsymbol{x} \in \Omega^{i}; \\ 
&\text { (d) For any } \varepsilon \in \mathbb{S}^{d}, \boldsymbol{x} \mapsto \mathcal{A}^{i}(\boldsymbol{x}, \boldsymbol{\varepsilon})\\
&~~~~~~~\text { is Lebesgue measurable on } \Omega^{i}; \\ 
&\text { (e) The mapping } \boldsymbol{x} \mapsto \mathcal{A}^{i}(\boldsymbol{x}, \mathbf{0}) \in Q^{i}.
\end{aligned}\right.
\end{equation}
And the contact functions $g^{i}_{j}$ ($i=1,2,\cdots,n$; $j=N$ or $T$) satisfies the following constraints:
\begin{equation}\label{2.13}
\left\{\begin{aligned}
&\text { (a) } g^{i}_{j}: \Gamma_{3}^{i} \times \mathbb{R} \rightarrow \mathbb{R}_{+}; \\
&\text { (b) There exists an } L_{j}^{i}>0 \text { such that } \\
&~~~~~~~~\left|g^{i}_{j}\left(\boldsymbol{x}, u_{1}\right)-g^{i}_{j}\left(\boldsymbol{x}, u_{2}\right)\right| \leqslant L_{j}^{i}\left|u_{1}-u_{2}\right| \\ 
&~~~~~~~~\forall u_{1}, u_{2} \in \mathbb{R} \text { a.e. in } \Omega^{i}; \\
&\text { (c) For any } u \in \mathbb{R}, \boldsymbol{x} \mapsto g^{i}_{j}(\boldsymbol{x}, u) \text { is measurable}; \\
&\text { (d) The mapping } \boldsymbol{x} \mapsto g^{i}_{j}(\boldsymbol{x}, 0) \in L^{2}\left(\Gamma_{3}^{i}\right) \text {. }
\end{aligned}\right.
\end{equation}
Note that the constraints on the contact functions here are rather trivial, the strictest of which is that $g^{i}_{j}$ is globally Lipschitz continuous with respect to the second variable. In fact, limited by its own material properties, the stress that an elastic layer can withstand is often limited, so this condition can be relaxed to local Lipschitz continuity. Therefore, it is easy to verify that the contact equations (\ref{2.8})-(\ref{2.10}) all satisfy the condition (\ref{2.13}).

Furthermore, the body force $\boldsymbol{f}_{0}^{i}$ and surface traction force $\boldsymbol{f}_{2}$ of the layered elastic system are assumed to satisfy the following conditions:
\begin{equation}\label{2.14}
\boldsymbol{f}_{0}^{i} \in \left(L^{2}(\Omega^{i})\right)^{2},~ 
\boldsymbol{f}_{2} \in \left(L^{2}\left(\Gamma_{2}^{1}\right)\right)^{2}.
\end{equation}

Next, $L(\boldsymbol{v})$ is defined as:
\begin{equation}\label{2.15}
L(\boldsymbol{v})=\sum_{i=1}^{n} \int_{\Omega^{i}} \boldsymbol{f}_{0}^{i} \cdot \boldsymbol{v}^{i} d x + \int_{\Gamma_{2}^{1}} \boldsymbol{f}_{2} \cdot \boldsymbol{v}^{1} d a \triangleq (\boldsymbol{f}, \boldsymbol{v})_{V},
\end{equation}
$\forall \boldsymbol{v} \in V$. Obviously, $L(\boldsymbol{v})$ is a linear functional about $\boldsymbol{v}$, and because $V$ is a real Hilbert space, the Riesz representation theorem guarantees $\boldsymbol{f} \in V$. Let $j: V \times V \rightarrow \mathbb{R}$ be the functional
\begin{equation}\label{2.16}
\begin{aligned}
j(\boldsymbol{v}, \boldsymbol{w})=&
\int_{\Gamma_{3}^{n}} g_{N}^{n}\left(v_{\beta}^{n}\right) w_{\beta}^{n}d a
+\int_{\Gamma_{3}^{n}} g_{T}^{n}\left(v_{\beta}^{n}\right)|\boldsymbol{w}_{\eta}^{n}| d a &&\\
&+ \sum_{i=1}^{n-1}\int_{\Gamma_{c}^{i}} g_{T}^{i}\left(\sigma_{N}^{i}(\boldsymbol{v}^{i})\right)|[\boldsymbol{w}_{T}^{i}]| d a && \forall \boldsymbol{v}, \boldsymbol{w} \in V.
\end{aligned}
\end{equation}
Furthermore, the double-variable operator $a^{i}(\cdot,\cdot): V^{i}\times V^{i} \to \mathbb{R}$ ($i=1,2,\ldots,n$) and $a(\cdot,\cdot): V\times V \to \mathbb{R}$ are defined as:
\begin{align}
a^{i}(\boldsymbol{v}^{i}, \boldsymbol{w}^{i})&=\int_{\Omega^{i}} \mathcal{A}^{i} \boldsymbol{\varepsilon}(\boldsymbol{v}^{i}): \boldsymbol{\varepsilon}(\boldsymbol{w}^{i}) \mathrm{d} x \nonumber\\
&= \left(\mathcal{A}^{i} \boldsymbol{\varepsilon}(\boldsymbol{v}^{i}), \boldsymbol{\varepsilon}(\boldsymbol{w}^{i})\right)_{L^{2}\left(\Omega^{i}\right)},\nonumber\\
& \triangleq \left(A^{i}\boldsymbol{v}^{i}, \boldsymbol{w}^{i}\right)_{V}\label{2.17}\\ 
a(\boldsymbol{v}, \boldsymbol{w})&=\sum_{i=1}^{n} a^{i}(\boldsymbol{v}^{i}, \boldsymbol{w}^{i}),\nonumber\\
& \triangleq \left(A\boldsymbol{v}, \boldsymbol{w}\right)_{V}\label{2.18}
\end{align}
$\forall \boldsymbol{v}, \boldsymbol{w} \in V$. Note that $A^{i}:V^{i}\to V^{i}$ is a nonlinear operator defined by $\mathcal{A}^{i}$. Thus operator $a(\cdot,\cdot)$ is nonlinear with respect to the first variable and linear with respect to the second variable.

Based on Green's formula in tensor form, the solution to the displacement field of problem $P_{0}$ can be transformed into a solution to a variational inequality problem of the form:
\begin{prm}[$P_{1}$]
Find a displacement $\boldsymbol{u}\in \mathcal{K}$ which satisfies:
\begin{equation}\label{2.19}
    a(\boldsymbol{u},\boldsymbol{v}-\boldsymbol{u}) + j(\boldsymbol{u},\boldsymbol{v}) - j(\boldsymbol{u},\boldsymbol{u}) \geqslant L(\boldsymbol{v}-\boldsymbol{u}), ~\forall \boldsymbol{v} \in \mathcal{K}.
\end{equation}
\end{prm}
The specific derivation process of Problem $P_{0}$ to Problem $P_{1}$ is presented in the Appendix. Under the framework of variational inequalities, the constraints of the original problem $P_{0}$ are greatly simplified.

\section{The existence and uniqueness of solution}
Note that the variational inequality (\ref{2.19}) in problem $P_{1}$ is equivalent to:
\begin{equation}\label{3.1}
    (A\boldsymbol{u},\boldsymbol{v}-\boldsymbol{u})_{V} + j(\boldsymbol{u},\boldsymbol{v}) - j(\boldsymbol{u},\boldsymbol{u}) \geqslant (\boldsymbol{f},\boldsymbol{v}-\boldsymbol{u})_{V}, ~\forall \boldsymbol{v} \in \mathcal{K},
\end{equation}
which is a variant of the elliptic variational inequality. Therefore, under certain assumptions about operators $A$ and $j$, the existence of a unique solution to the variational inequality (\ref{3.1}) can be proved. Then it is only necessary to prove that operators $A$ and $j$ defined by operators $\mathcal{A}^{i}$ and $g_{r}^{i}$ satisfy these conditions.

First, we assume that $A:V \to V$ is a strongly monotonic and Lipschitz continuous operator, that is, it satisfies:
\begin{equation}\label{3.2}
\left.\begin{aligned}
&\text { (a) } \exists M>0 \text { such that } \\
&~~~~~~\left(A \boldsymbol{v}_{1}-A \boldsymbol{v}_{2}, \boldsymbol{v}_{1}-\boldsymbol{v}_{2}\right)_{V} \geqslant M\left\|\boldsymbol{v}_{1}-\boldsymbol{v}_{2}\right\|_{V}^{2},~ \forall \boldsymbol{v}_{1}, \boldsymbol{v}_{2} \in V ; \\
&\text { (b) } \exists L_{A}>0 \text { such that } \\
&~~~~~~\left\|A \boldsymbol{v}_{1}-A \boldsymbol{v}_{2}\right\|_{V} \leqslant L_{A}\left\|\boldsymbol{v}_{1}-\boldsymbol{v}_{2}\right\|_{V},~ \forall \boldsymbol{v}_{1}, \boldsymbol{v}_{2} \in V .
\end{aligned}\right\}
\end{equation}
Then the functional $j$ satisfies the following conditions:
\begin{equation}\label{3.3}
\left.\begin{aligned}
&\text { (a) } \forall \boldsymbol{w} \in V, j(\boldsymbol{w}, \cdot) \text { is convex and l.s.c on } V \text {; } \\
&\text { (b) } \exists m>0 \text { such that: } \\
&~~~~~~~j\left(\boldsymbol{w}_{1}, \boldsymbol{v}_{2}\right) - j\left(\boldsymbol{w}_{1}, \boldsymbol{v}_{1}\right) - \left( j\left(\boldsymbol{w}_{2}, \boldsymbol{v}_{2}\right) - j\left(\boldsymbol{w}_{2}, \boldsymbol{v}_{1}\right) \right)\\
&~~~~~~~\leqslant m\left\|\boldsymbol{w}_{1}-\boldsymbol{w}_{2}\right\|_{V} \left\|\boldsymbol{v}_{1}-\boldsymbol{v}_{2}\right\|_{V}, ~ \forall \boldsymbol{w}_{1}, \boldsymbol{w}_{2}, \boldsymbol{v}_{1}, \boldsymbol{v}_{2} \in V.
\end{aligned}\right\}
\end{equation}
where l.s.c represents lower semicontinuous.

Based on the above assumptions and Banach's fixed point theorem, which is often used to deal with similar problems\cite{han2000evolutionary}, we have the following theorem:
\begin{thm}
Let (\ref{3.2})-(\ref{3.3}) hold and $\boldsymbol{f}\in V$. Then, if $M>m$, there exists a unique solution $\boldsymbol{u}\in \mathcal{K}$ to the problem (\ref{3.1}).
\end{thm}
\begin{proof}
Let $\boldsymbol{p}\in \mathcal{K}$ be given. Consider the following variational inequality:
\begin{equation}\label{3.4}
    (A\boldsymbol{u},\boldsymbol{v}-\boldsymbol{u})_{V} + j(\boldsymbol{p},\boldsymbol{v}) - j(\boldsymbol{p},\boldsymbol{u}) \geqslant (\boldsymbol{f},\boldsymbol{v}-\boldsymbol{u})_{V}, ~\forall \boldsymbol{v} \in \mathcal{K},
\end{equation}
which has a unique solution according to the theory of elliptic variational fractional inequalities\cite{capatina2014variational}. For each $\boldsymbol{p}\in \mathcal{K}$, the operator $\Lambda: \mathcal{K}\to \mathcal{K}$ can be defined as:
\begin{equation}\label{3.5}
\Lambda \boldsymbol{p} = \boldsymbol{u}, ~ \forall \boldsymbol{p} \in \mathcal{K}.
\end{equation}
Then, we denote $\boldsymbol{u}_{1} = \Lambda \boldsymbol{p}_{1}$ and $\boldsymbol{u}_{2} = \Lambda \boldsymbol{p}_{2}$. Based on the elliptic variational inequality (\ref{3.4}), it is easy to deduce:
\begin{equation*}
\left(A \boldsymbol{u}_{1}-A \boldsymbol{u}_{2}, \boldsymbol{u}_{1} - \boldsymbol{u}_{2}\right) \leqslant j\left(\boldsymbol{p}_{2}, \boldsymbol{u}_{1}\right) - j\left(\boldsymbol{p}_{2}, \boldsymbol{u}_{2}\right) - \left[j\left(\boldsymbol{p}_{1}, \boldsymbol{u}_{1}\right) - j\left(\boldsymbol{p}_{1}, \boldsymbol{u}_{2}\right)\right]
\end{equation*}
Then, using properties (\ref{3.2}) and (\ref{3.3}), it is easy to obtain:
\begin{equation}
\begin{aligned}
&M\left\|\boldsymbol{u}_{1} - \boldsymbol{u}_{2}\right\|_{V} \leqslant m\left\|\boldsymbol{p}_{1}-\boldsymbol{p}_{2}\right\|_{V}\\ 
\Rightarrow &\left\|\Lambda\boldsymbol{p}_{1} - \Lambda\boldsymbol{p}_{2}\right\|_{V} \leqslant \frac{m}{M}\left\|\boldsymbol{p}_{1}-\boldsymbol{p}_{2}\right\|_{V}.
\end{aligned}
\end{equation}
Therefore, when $M>m$, the operator $\Lambda$ is a condensing map in the Banach space $\mathcal{K}$. Based on Banach's fixed point theorem, there exists a unique $\boldsymbol{u}^{*}$ such that the following holds:
\begin{equation*}
\Lambda \boldsymbol{u}^{*} = \boldsymbol{u}^{*}, ~ \forall \boldsymbol{p} \in \mathcal{K}.
\end{equation*}
Taking $\boldsymbol{p}=\boldsymbol{u}^{*}$ in (\ref{3.4}), we have:
\begin{equation*}
(A \boldsymbol{u}^{*}, \boldsymbol{v}-\boldsymbol{u}^{*})_{V}+j(\boldsymbol{u}^{*}, \boldsymbol{v})-j(\boldsymbol{u}^{*}, \boldsymbol{u}^{*}) \geqslant(\boldsymbol{f}, \boldsymbol{v}-\boldsymbol{u}^{*})_{V}, \forall \boldsymbol{v} \in \mathcal{K}.
\end{equation*}
That is, $\boldsymbol{u}^{*}$ is the solution of the variational inequality \ref{3.1}. The uniqueness of the solution of the elliptic variational inequality and the Banach fixed point theorem guarantee the uniqueness of $\boldsymbol{u}^{*}$.
\end{proof}

Then, in order to prove the existence and uniqueness of the solution to problem $P_{1}$, it is only necessary to verify that the operator $A$ defined by equation (\ref{2.18}) and the functional $j(\cdot,\cdot)$ defined by equation (\ref{2.16}) satisfy conditions (\ref{3.2}) and (\ref{3.3}), respectively. From this, the following theorem holds:
\begin{thm}
Assume that (\ref{2.12})-(\ref{2.14}) hold. Then there exists $M>0$ and $m>0$, which depend only on $\Omega^{i}$, $\Gamma^{i}$, $g^{i}_{j}$ and $\mathcal{A}^{i}$ ($j=N$ or $T$, $i=1,2,\ldots,n$), such that the problem $P_{1}$ has a unique solution $\boldsymbol{u}\in \mathcal{K}$ if $M>m$.
\end{thm}
\begin{proof}
According to (\ref{2.17}), (\ref{2.18}) and (\ref{2.12}.c), the following inequality relation holds:
\begin{equation*}
\begin{aligned}
&\left(A \boldsymbol{u}_{1} - A \boldsymbol{u}_{2}, \boldsymbol{u}_{1} - \boldsymbol{u}_{2}\right)_{V} \\ =&\sum_{i=1}^{n} \int_{\Omega^{i}} \left( \mathcal{A}^{i} \left(\boldsymbol{\varepsilon}\left(\boldsymbol{u}_{1}^{i}\right)\right) - \mathcal{A}^{i} \left(\boldsymbol{\varepsilon}\left(\boldsymbol{u}_{2}^{i}\right)\right) \right): \left(\boldsymbol{\varepsilon}\left(\boldsymbol{u}_{1}^{i}\right) - \boldsymbol{\varepsilon}\left(\boldsymbol{u}_{2}^{i}\right)\right) d x \\
\geqslant & \sum_{i=1}^{m} M^{i}\left\|\boldsymbol{\varepsilon}\left(\boldsymbol{u}_{1}^{i}\right) - \boldsymbol{\varepsilon}\left(\boldsymbol{u}_{2}^{i}\right)\right\|^{2}_{L^{2}\left(\Omega^{i}\right)^{d}}\\
\geqslant & M\|\boldsymbol{u}_{1} - \boldsymbol{u}_{2}\|^{2}_{V}
\end{aligned}
\end{equation*}
where $M=\min(M^{1},M^{2},\ldots,M^{n})$. Therefore, $A$ satisfies strong monotonicity. And the Lipschitz continuity of operator $\mathcal{A}^{i}$ for the second variable also ensures that operator $A$ is Lipschitz continuous. Hence, it is proved that $A$ satisfies the condition (\ref{3.2}).

According to the formula (\ref{2.16}), it is easy to verify that $j(\cdot,\cdot)$ satisfies the condition (\ref{3.3}(a)), and it is only necessary to verify that the condition (\ref{3.3}(b)) holds. Substitute (\ref{2.16}) into (\ref{3.3}(b)) to obtain the following relationship:
\begin{align}\label{3.7}
& j\left(\boldsymbol{w}_{1}, \boldsymbol{v}_{2}\right) - j\left(\boldsymbol{w}_{1}, \boldsymbol{v}_{1}\right) - \left( j\left(\boldsymbol{w}_{2}, \boldsymbol{v}_{2}\right) - j\left(\boldsymbol{w}_{2}, \boldsymbol{v}_{1}\right) \right) \nonumber\\
=& \int_{\Gamma_{3}^{n}} g_{N}^{n}\left(w_{1\beta}^{n}\right) v_{2\beta}^{n} - g_{N}^{n}\left(w_{1\beta}^{n}\right) v_{1\beta}^{n} - \left(g_{N}^{n}\left(w_{2\beta}^{n}\right) v_{2\beta}^{n} - g_{N}^{n}\left(w_{2\beta}^{n}\right) v_{1\beta}^{n} \right) da \nonumber\\ 
& + \int_{\Gamma_{3}^{n}} g_{T}^{n}\left(w_{1\beta}^{n}\right)|\boldsymbol{v}_{2\eta}^{i}| - g_{T}^{n}\left(w_{1\beta}^{n}\right)|\boldsymbol{v}_{1\eta}^{i}| - \left( g_{T}^{n}\left(w_{2\beta}^{n}\right)|\boldsymbol{v}_{2\eta}^{i}| - g_{T}^{n}\left(w_{2\beta}^{n}\right)|\boldsymbol{v}_{1\eta}^{i}| \right) d a \nonumber\\
& + \sum_{i=1}^{n-1}\int_{\Gamma_{c}^{i}} g_{T}^{i}\left(\sigma_{N}^{i}(\boldsymbol{w}^{i}_{1})\right)|[\boldsymbol{v}_{2T}^{i}]| - g_{T}^{i}\left(\sigma_{N}^{i}(\boldsymbol{w}^{i}_{1})\right)|[\boldsymbol{v}_{1T}^{i}]| \nonumber\\ 
& - \left( g_{T}^{i}\left(\sigma_{N}^{i}(\boldsymbol{w}^{i}_{2})\right)|[\boldsymbol{v}_{2T}^{i}]| - g_{T}^{i}\left(\sigma_{N}^{i}(\boldsymbol{w}^{i}_{2})\right)|[\boldsymbol{v}_{1T}^{i}]| \right) d a.
\end{align}
Based on (\ref{2.0}) and (\ref{2.13}(b)), the first term on the right-hand side of (\ref{3.7}) satisfies:
\begin{align}\label{3.8}
&\int_{\Gamma_{3}^{n}} g_{N}^{n}\left(w_{1\beta}^{n}\right) v_{2\beta}^{n} - g_{N}^{n}\left(w_{1\beta}^{n}\right) v_{1\beta}^{n} - \left(g_{N}^{n}\left(w_{2\beta}^{n}\right) v_{2\beta}^{n} - g_{N}^{n}\left(w_{2\beta}^{n}\right) v_{1\beta}^{n} \right) da \nonumber\\
\leqslant & \int_{\Gamma_{3}^{n}} |g_{N}^{n}\left(w_{1\beta}^{n}\right) -  g_{N}^{n}\left(w_{2\beta}^{n}\right)|\cdot |v_{2\beta}^{n} - v_{1\beta}^{n}| da \nonumber\\
\leqslant & L_{N}^{n} \int_{\Gamma_{3}^{n}} |w_{1\beta}^{n} - w_{2\beta}^{n}|\cdot |v_{2\beta}^{n} - v_{1\beta}^{n}| da \nonumber\\
\leqslant & L_{N}^{n} c_{t}^{n} \|\boldsymbol{w}_{1}^{n} - \boldsymbol{w}_{2}^{n}\|_{V^{n}} \cdot \|\boldsymbol{v}_{1}^{n} - \boldsymbol{v}_{2}^{n}\|_{V^{n}}.
\end{align}
Similarly, it is easy to prove that the second term on the right-hand side of (\ref{3.7}) satisfies:
\begin{align}\label{3.9}
& \int_{\Gamma_{3}^{n}} g_{T}^{n}\left(w_{1\beta}^{n}\right)|\boldsymbol{v}_{2\eta}^{i}| - g_{T}^{n}\left(w_{1\beta}^{n}\right)|\boldsymbol{v}_{1\eta}^{i}| - \left( g_{T}^{n}\left(w_{2\beta}^{n}\right)|\boldsymbol{v}_{2\eta}^{i}| - g_{T}^{n}\left(w_{2\beta}^{n}\right)|\boldsymbol{v}_{1\eta}^{i}| \right) d a \nonumber\\
\leqslant & L_{T}^{n} c_{t}^{n} \|\boldsymbol{w}_{1}^{n} - \boldsymbol{w}_{2}^{n}\|_{V^{n}} \cdot \|\boldsymbol{v}_{1}^{n} - \boldsymbol{v}_{2}^{n}\|_{V^{n}}.
\end{align}
Then, it is easy to deduce that the third term on the right-hand side of equation (\ref{3.7}) satisfies:
\begin{align}\label{3.10}
& \sum_{i=1}^{n-1}\int_{\Gamma_{c}^{i}} g_{T}^{i}\left(\sigma_{N}^{i}(\boldsymbol{w}^{i}_{1})\right)|[\boldsymbol{v}_{2T}^{i}]| - g_{T}^{i}\left(\sigma_{N}^{i}(\boldsymbol{w}^{i}_{1})\right)|[\boldsymbol{v}_{1T}^{i}]| \nonumber\\ 
& - \left( g_{T}^{i}\left(\sigma_{N}^{i}(\boldsymbol{w}^{i}_{2})\right)|[\boldsymbol{v}_{2T}^{i}]| - g_{T}^{i}\left(\sigma_{N}^{i}(\boldsymbol{w}^{i}_{2})\right)|[\boldsymbol{v}_{1T}^{i}]| \right) d a \nonumber\\ 
\leqslant & \sum_{i=1}^{n-1} L_{T}^{i} \int_{\Gamma_{c}^{i}} \left|\sigma_{N}^{i}(\boldsymbol{w}^{i}_{1}) - \sigma_{N}^{i}(\boldsymbol{w}^{i}_{2})\right|\cdot \left||[\boldsymbol{v}_{1T}^{i}]| - |[\boldsymbol{v}_{2T}^{i}]|\right| da .
\end{align}
According to (\ref{2.12}(b)), the following inequality relation holds:
\begin{align*}
\left|\sigma_{N}^{i}(\boldsymbol{w}^{i}_{1}) - \sigma_{N}^{i}(\boldsymbol{w}^{i}_{2})\right| &\leqslant 
\left|\mathcal{A}^{i}\left(\boldsymbol{\varepsilon}\left(\boldsymbol{w}_{1}^{i}\right)\right)-\mathcal{A}^{i}\left(\boldsymbol{\varepsilon}\left(\boldsymbol{w}_{2}^{i}\right)\right)\right|\\
& \leqslant L_{A}^{i} \left| \boldsymbol{\varepsilon}\left(\boldsymbol{w}_{1}^{i}\right) - \boldsymbol{\varepsilon}\left(\boldsymbol{w}_{2}^{i}\right) \right|.
\end{align*}
In addition, based on the definition of the jump operator, the following relationship can be derived:
\begin{align*}
&\left||[\boldsymbol{v}_{1T}^{i}]| - |[\boldsymbol{v}_{2T}^{i}]|\right| \leqslant \left|[\boldsymbol{v}_{1T}^{i}] - [\boldsymbol{v}_{2T}^{i}]\right| = \left|(\boldsymbol{v}_{1\eta}^{i} - \boldsymbol{v}_{2\eta}^{i}) - (\boldsymbol{v}_{1\tau}^{i+1} - \boldsymbol{v}_{2\tau}^{i+1})\right|\\
\leqslant & \left|\boldsymbol{v}_{1\eta}^{i} - \boldsymbol{v}_{2\eta}^{i}\right| + \left| \boldsymbol{v}_{1\tau}^{i+1} - \boldsymbol{v}_{2\tau}^{i+1}\right| \leqslant \left|\boldsymbol{v}_{1}^{i} - \boldsymbol{v}_{2}^{i}\right| + \left| \boldsymbol{v}_{1}^{i+1} - \boldsymbol{v}_{2}^{i+1}\right|
\end{align*}
Based on the above two inequalities, inequality (\ref{2.0})-(\ref{2.01}) and Hölder inequality, equation (\ref{3.10}) satisfies the following inequality relation:
\begin{align}\label{3.11}
& \sum_{i=1}^{n-1}\int_{\Gamma_{c}^{i}} g_{T}^{i}\left(\sigma_{N}^{i}(\boldsymbol{w}^{i}_{1})\right)|[\boldsymbol{v}_{2T}^{i}]| - g_{T}^{i}\left(\sigma_{N}^{i}(\boldsymbol{w}^{i}_{1})\right)|[\boldsymbol{v}_{1T}^{i}]| \nonumber\\ 
& - \left( g_{T}^{i}\left(\sigma_{N}^{i}(\boldsymbol{w}^{i}_{2})\right)|[\boldsymbol{v}_{2T}^{i}]| - g_{T}^{i}\left(\sigma_{N}^{i}(\boldsymbol{w}^{i}_{2})\right)|[\boldsymbol{v}_{1T}^{i}]| \right) d a \nonumber\\ 
\leqslant & \sum_{i=1}^{n-1} L_{T}^{i} L_{A}^{i} \int_{\Gamma_{c}^{i}} \left| \boldsymbol{\varepsilon}\left(\boldsymbol{w}_{1}^{i}\right) - \boldsymbol{\varepsilon}\left(\boldsymbol{w}_{2}^{i}\right) \right| \cdot \left(\left|\boldsymbol{v}_{1}^{i} - \boldsymbol{v}_{2}^{i}\right| + \left| \boldsymbol{v}_{1}^{i+1} - \boldsymbol{v}_{2}^{i+1}\right| \right) da \nonumber\\
\leqslant & \sum_{i=1}^{n} L_{T}^{i} L_{A}^{i} c_{h}^{i} c_{t}^{i} \| \boldsymbol{w}_{1}^{i} - \boldsymbol{w}_{2}^{i}\|_{V^{i}} \cdot \| \boldsymbol{v}_{1}^{i} - \boldsymbol{v}_{2}^{i}\|_{V^{i}}.
\end{align}
Finally, by substituting inequalities (\ref{3.8}), (\ref{3.9}) and (\ref{3.11}) into (\ref{3.7}), the following inequalities hold:
\begin{align}
& j\left(\boldsymbol{w}_{1}, \boldsymbol{v}_{2}\right) - j\left(\boldsymbol{w}_{1}, \boldsymbol{v}_{1}\right) - \left( j\left(\boldsymbol{w}_{2}, \boldsymbol{v}_{2}\right) - j\left(\boldsymbol{w}_{2}, \boldsymbol{v}_{1}\right) \right) \nonumber\\
\leqslant & \sum_{i=1}^{n-1} L_{T}^{i} L_{A}^{i} c_{h}^{i} c_{t}^{i} \| \boldsymbol{w}_{1}^{i} - \boldsymbol{w}_{2}^{i}\|_{V^{i}} \cdot \| \boldsymbol{v}_{1}^{i} - \boldsymbol{v}_{2}^{i}\|_{V^{i}} \nonumber \\ 
&+ \left(L_{N}^{n} c_{t}^{n} + L_{T}^{n} c_{t}^{n} + L_{T}^{n} L_{A}^{n} c_{h}^{n} c_{t}^{n}\right) \cdot \|\boldsymbol{w}_{1}^{n} - \boldsymbol{w}_{2}^{n}\|_{V^{n}} \cdot \|\boldsymbol{v}_{1}^{n} - \boldsymbol{v}_{2}^{n}\|_{V^{n}} \nonumber\\
\leqslant & m \cdot \|\boldsymbol{w}_{1} - \boldsymbol{w}_{2}\|_{V} \cdot \|\boldsymbol{v}_{1} - \boldsymbol{v}_{2}\|_{V},
\end{align}
where 
$$
m = \max(L_{T}^{1} L_{A}^{1} c_{h}^{1} c_{t}^{1}, \ldots, L_{T}^{n-1} L_{A}^{n-1} c_{h}^{n-1} c_{t}^{n-1}, L_{N}^{n} c_{t}^{n} + L_{T}^{n} c_{t}^{n} + L_{T}^{n} L_{A}^{n} c_{h}^{n} c_{t}^{n}).
$$

Therefore, the operator $j(\cdot,\cdot)$ satisfies the condition (\ref{3.3}), and then applying Theorem 3.1, Theorem 3.2 is proved.
\end{proof}

Hence, the existence and uniqueness theorem of the solution to problem $P_{1}$ is proved. 
\begin{rem}
It is worth noting that whether the critical condition $M>m$ holds will depend on the parameters of the specific physical model. In particular, when $L_{j}^{i}$ is sufficiently small in the friction condition (\ref{2.13}), $M>m$ can always be guaranteed, that is, the solution exists and is unique. Therefore, when the problem $P_{0}$ has a unique solution, $P_{0}$ and $P_{1}$ must be equivalent, which means that the variational inequality can simplify the constraints without losing the constraint information of the original system of equations.
\end{rem}

\section{Approximation and Convergence}
After the existence and uniqueness of the solution of problem $P_{1}$ has been proved, the approximate solution obtained by the finite element method will be studied in this section. Then, after the approximate relation is obtained, the conditions under which the discrete approximate solution converges to the true solution will also be analyzed.

\subsection{Discrete approximation}

Let $\mathcal{K}^{h}\subset\mathcal{K}$ be the finite-dimensional function space constructed by the finite element method. Then the approximation problem in space $\mathcal{K}^{h}$ is described as follows:
\begin{prm}[$P_{1}^{h}$]
Find a displacement $\boldsymbol{u}^{h}\in \mathcal{K}^{h}$ which satisfies:
\begin{equation}\label{4.1}
    (A\boldsymbol{u}^{h},\boldsymbol{v}^{h}-\boldsymbol{u}^{h})_{V} + j(\boldsymbol{u}^{h},\boldsymbol{v}^{h}) - j(\boldsymbol{u}^{h},\boldsymbol{u}^{h}) \geqslant (\boldsymbol{f},\boldsymbol{v}^{h}-\boldsymbol{u}^{h})_{V}, ~\forall \boldsymbol{v}^{h} \in \mathcal{K}^{h},
\end{equation}
\end{prm}
Similar to the proof of Theorem 3.2 in Section 3, it is easy to prove that there is a unique solution $\boldsymbol{u}^{h}\in \mathcal{K}^{h}$ to problem $P_{1}^{h}$ under certain conditions. 
Then, the main purpose of this subsection is to discuss the error between the approximate solution and the true solution: $\boldsymbol{u} - \boldsymbol{u}^{h}$.

Due to $\mathcal{K}^{h}\subset \mathcal{K}$, by taking $\boldsymbol{v}=\boldsymbol{u}^{h}$ in inequality (\ref{3.1}), it can be obtained:
\begin{equation}\label{4.2}
(A\boldsymbol{u},\boldsymbol{u}^{h}-\boldsymbol{u})_{V} + j(\boldsymbol{u},\boldsymbol{u}^{h}) - j(\boldsymbol{u},\boldsymbol{u}) \geqslant (\boldsymbol{f},\boldsymbol{u}^{h}-\boldsymbol{u})_{V}.
\end{equation}
Then, after summing and manipulating the inequalities (\ref{4.1}) and (\ref{4.2}), the following inequalities can be obtained:
\begin{align}\label{4.3}
(A\boldsymbol{u} - A\boldsymbol{u}^{h} , \boldsymbol{u} -  \boldsymbol{u}^{h})_{V} \leqslant & (A\boldsymbol{u}^{h},\boldsymbol{v}^{h} -\boldsymbol{u})_{V} + (\boldsymbol{f},\boldsymbol{u} - \boldsymbol{v}^{h} )_{V} \nonumber\\ 
&+ j(\boldsymbol{u},\boldsymbol{u}^{h}) - j(\boldsymbol{u},\boldsymbol{u}) + j(\boldsymbol{u}^{h},\boldsymbol{v}^{h}) - j(\boldsymbol{u}^{h},\boldsymbol{u}^{h})\nonumber\\ 
\leqslant & (A\boldsymbol{u}^{h} - A\boldsymbol{u} ,\boldsymbol{v}^{h} -\boldsymbol{u})_{V} + j(\boldsymbol{u},\boldsymbol{u}^{h}) - j(\boldsymbol{u},\boldsymbol{v}^{h}) \nonumber \\
&- \left[j(\boldsymbol{u}^{h},\boldsymbol{u}^{h}) - j(\boldsymbol{u}^{h},\boldsymbol{v}^{h}) \right] + R(\boldsymbol{u} ,\boldsymbol{v}^{h}),
\end{align}
where 
\begin{equation}\label{4.4}
R(\boldsymbol{u},\boldsymbol{v}^{h}) = (A\boldsymbol{u},\boldsymbol{v}^{h} -\boldsymbol{u})_{V} + (\boldsymbol{f},\boldsymbol{u} - \boldsymbol{v}^{h} )_{V} - j(\boldsymbol{u},\boldsymbol{u}) + j(\boldsymbol{u},\boldsymbol{v}^{h}).
\end{equation}
Using conditions (\ref{3.2}) and (\ref{3.3}), it can be deduced form (\ref{4.3}) that:
\begin{align*}
& M\|\boldsymbol{u} -\boldsymbol{u}^{h}\|_{V}^{2} \\ \leqslant & L_{A} \|\boldsymbol{u} -\boldsymbol{u}^{h}\|_{V} \cdot \|\boldsymbol{u} -\boldsymbol{v}^{h}\|_{V} + m \|\boldsymbol{u} -\boldsymbol{u}^{h}\|_{V} \cdot \|\boldsymbol{v}^{h} -\boldsymbol{u}^{h}\|_{V} + |R(\boldsymbol{u},\boldsymbol{v}^{h})|\\
\leqslant & (L_{A} + m) \|\boldsymbol{u} -\boldsymbol{u}^{h}\|_{V} \cdot \|\boldsymbol{u} -\boldsymbol{v}^{h}\|_{V} + m \|\boldsymbol{u} -\boldsymbol{u}^{h}\|_{V}^{2} + |R(\boldsymbol{u},\boldsymbol{v}^{h})|.
\end{align*}
Based on the condition $M>m$, the following inequality can be deduced from the above formula:
\begin{align*}
\|\boldsymbol{u} -\boldsymbol{u}^{h}\|_{V} \leqslant c\left( \|\boldsymbol{u} -\boldsymbol{v}^{h}\|_{V} + |R(\boldsymbol{u},\boldsymbol{v}^{h})|^{\frac{1}{2}} \right).
\end{align*}
Note that this inequality holds for any $\boldsymbol{v}^{h} \in \mathcal{K}^{h}$, so it can be summarized as the following theorem:
\begin{thm}
Assume the conditions (\ref{3.2})-(\ref{3.3}) hold, $\boldsymbol{f}\in V$ and $M>m$. Then, the error between the solution $\boldsymbol{u}$ of problem $P_{1}$ and the solution $\boldsymbol{u}^{h}$ of problem $P_{1}^{h}$ has the following estimation: 
\begin{align}\label{4.5}
\|\boldsymbol{u} -\boldsymbol{u}^{h}\|_{V} \leqslant c \inf_{\boldsymbol{v}^{h}\in \mathcal{K}^{h}}\left( \|\boldsymbol{u} -\boldsymbol{v}^{h}\|_{V} + |R(\boldsymbol{u},\boldsymbol{v}^{h})|^{\frac{1}{2}} \right),
\end{align}
where $|R(\boldsymbol{u},\boldsymbol{v}^{h})|$ is defined in (\ref{4.4}).
\end{thm}
Therefore, in order to judge whether the discrete approximate solution obtained by the finite element method converges, it is only necessary to verify whether the inequality (\ref{4.5}) tends to $0$.

\subsection{Convergence Analysis}

After obtaining the error estimate of the solution, in this subsection the convergence of the numerical solution will be analyzed. Note that the regularity of the solution to problem $P_{1}$ is unknown, and the probability that the solution does not possess a high degree of regularity is higher in the actual physical model. Therefore, it is particularly critical to study the convergence of numerical solutions without resorting to the high regularity assumption of the weak solution of $P_{1}$. Since the error estimation (\ref{4.5}) only holds under certain regularity assumptions, the following analysis is only based on the basic regularity condition $\boldsymbol{u}\in\mathcal{K}$. 
Then, we introduce two basic hypotheses from \cite{han2000evolutionary}:
\begin{hyp}[$H_{1}$]
There exists a dense subset $\mathcal{K}_0 \subset \mathcal{K}$ and a function $\gamma(h)>0$ which satisfy $\gamma(h) \to 0$ as $h\to 0^{+}$ such that
$$
\left\|\boldsymbol{v}-\mathcal{P}^{h} \boldsymbol{v}\right\|_{V} \leqslant \gamma(h)\|\boldsymbol{v}\|_{V},~~ \forall \boldsymbol{v} \in \mathcal{K}_0,
$$
where $\mathcal{P}^{h}: \mathcal{K}\to \mathcal{K}^{h}$ represents a projection operator defined by
$$
\left\|\boldsymbol{v}-\mathcal{P}^{h} \boldsymbol{v}\right\|_{V}=\inf_{\boldsymbol{v}^{h} \in V^{h}}\left\|\boldsymbol{v}-\boldsymbol{v}^{h}\right\|_{V}, \quad \boldsymbol{v} \in \mathcal{K}.
$$
\end{hyp}
Based on the fact that $\mathcal{K}$ is a Hilbert space and $\mathcal{K}^{h}$ is a finite-dimensional space, it can be deduced that $\mathcal{P}^{h}$ is well-defined, linear and non-expansive\cite{brenner2008mathematical}, where non-expansion means:
$$
\left\|\mathcal{P}^{h} \boldsymbol{v}_{1}-\mathcal{P}^{h} \boldsymbol{v}_{2}\right\|_{V} \leqslant \left\|\boldsymbol{v}_{1}-\boldsymbol{v}_{2}\right\|_{V} \quad \forall \boldsymbol{v}_{1}, \boldsymbol{v}_{2} \in \mathcal{K}.
$$
\begin{hyp}[$H_{2}$]
Let $S\subset \mathcal{K}$ be any bounded set in $\mathcal{K}$ and in $S\times \mathcal{K}$ the functional $j(\cdot,\cdot)$ is uniformly continuous with respect to the second argument, that is: $\forall \varepsilon>0, \exists \delta=\delta(S)>0$ such that $$
\left|j\left(\boldsymbol{w}, \boldsymbol{v}_{1}\right)-j\left(\boldsymbol{w}, \boldsymbol{v}_{2}\right)\right| < \varepsilon,
$$
$\forall \boldsymbol{w} \in S$, $\forall \boldsymbol{v}_{1}, \boldsymbol{v}_{2} \in \mathcal{K}$ with $\left\|\boldsymbol{v}_{1} - \boldsymbol{v}_{2}\right\|_{V}<\delta$.
\end{hyp}
Under the above two assumptions, the convergence of the numerical solution can be analyzed. First, the following lemmas hold:
\begin{lem}
In the case where Hypothesis $H_{1}$ holds, for any $\boldsymbol{v}\in\mathcal{K}$, the following convergence holds:
$$
\lim_{h\to 0^{+}}\left\|\boldsymbol{v}-\mathcal{P}^{h} \boldsymbol{v}\right\|_{V} = 0.
$$
\end{lem}
\begin{proof}
When $\boldsymbol{v}=0$, the lemma holds. Then let $\boldsymbol{v}\ne 0$. Since $\mathcal{K}_{0}$ is dense in $\mathcal{K}$, for any $\varepsilon>0$, we can find a $\boldsymbol{w}\in\mathcal{K}_{0}$ such that: 
$$
\|\boldsymbol{w}-\boldsymbol{v}\|_{V}<\frac{\varepsilon}{4}.
$$
Then, based on Hypothesis $H_{1}$, it can be known that:
$$
\left\|\boldsymbol{w}-\mathcal{P}^{h} \boldsymbol{w}\right\|_{V} \leqslant \gamma(h)\|\boldsymbol{w}\|_{V}.
$$
Since $\lim_{h\to 0^{+}}\gamma(h)=0$, there exist $h_{0}>0$ such that $\gamma(h) < \frac{\varepsilon}{2\|\boldsymbol{w}\|_{V}}$ for any $0<h<h_{0}$, that is 
$$
\left\|\boldsymbol{w}-\mathcal{P}^{h} \boldsymbol{w}\right\|_{V} <\frac{\varepsilon}{2}.
$$
Finally, from the triangle inequality and the non-expansion of the projection operator $\mathcal{P}^{h}$, we have:
\begin{align*}
\left\|\boldsymbol{v}-\mathcal{P}^{h} \boldsymbol{v}\right\|_{V} &=\left\|(\boldsymbol{v}-\boldsymbol{w}) - \mathcal{P}^{h}(\boldsymbol{v} - \boldsymbol{w}) + \left(\boldsymbol{w} - \mathcal{P}^{h} \boldsymbol{w}\right)\right\|_{V} \\
& \leqslant 2\|\boldsymbol{v} - \boldsymbol{w}\|_{V} + \left\|\boldsymbol{w} - \mathcal{P}^{h} \boldsymbol{w}\right\|_{V} \\
&\leqslant \varepsilon
\end{align*}
\end{proof}

Based on Lemma 4.2 and Hypothesis $H_{2}$, the convergence of the numerical solution has the following conclusions:
\begin{thm}
Assume the conditions (\ref{3.2})-(\ref{3.3}), hypotheses $H_{1}$, $H_{2}$ hold, $\boldsymbol{f}\in V$ and $M>m$. Then the solution of the problem $P^{h}_{1}$ has the following convergence:
$$
\left\|\boldsymbol{u}-\boldsymbol{u}^{h}\right\|_{V} \rightarrow 0 ~ \text { as } ~h \rightarrow 0.
$$
\end{thm}
\begin{proof}
Before using Theorem 4.1, the upper bound of term $R\left(\boldsymbol{u}, \boldsymbol{v}^{h}\right)$ needs to be determined. Based on definition (\ref{4.4}), the properties of operators $A$, $\boldsymbol{f}$ and solution $\boldsymbol{u}$, the following inequalities hold:
$$
R\left(\boldsymbol{u}, \boldsymbol{v}^{h}\right) \leqslant c\left\|\boldsymbol{u} - \boldsymbol{v}^{h}\right\|_{V} + \left|j(\boldsymbol{u}, \boldsymbol{u}) - j\left(\boldsymbol{u}, \boldsymbol{v}^{h}\right)\right|,
$$
where the constant c depends only on $L_{A}$, $\|\boldsymbol{f}\|_{V}$ and $\|\boldsymbol{u}\|_{V}$. Therefore, taking $\boldsymbol{v}^{h} = \mathcal{P}^{h} \boldsymbol{u}$, from Theorem 4.1 and the above inequality, it is easy to deduce the following upper bound of error:
\begin{align*}
\|\boldsymbol{u} -\boldsymbol{u}^{h}\|_{V} \leqslant c \left( \|\boldsymbol{u} - \mathcal{P}^{h} \boldsymbol{u}\|_{V} + \left\|\boldsymbol{u} - \mathcal{P}^{h} \boldsymbol{u}\right\|_{V}^{\frac{1}{2}} + \left|j(\boldsymbol{u}, \boldsymbol{u}) - j\left(\boldsymbol{u}, \mathcal{P}^{h} \boldsymbol{u}\right)\right|^{\frac{1}{2}} \right).
\end{align*}
Lemma 4.2 has shown that:
$$
\lim_{h \rightarrow 0^{+}}\left\|\boldsymbol{u}-\mathcal{P}^{h} \boldsymbol{u}\right\|_{V}=0.
$$
Combined with the hypotheses $H_{2}$, $j(\cdot, \cdot)$ converges uniformly with respect to the second variable, it is easy to deduce:
$$
\lim_{h \rightarrow 0^{+}}\left|j(\boldsymbol{u}, \boldsymbol{u}) - j(\boldsymbol{u}, \mathcal{P}^{h}\boldsymbol{u})\right|=0.
$$
Therefore, $\lim_{h \rightarrow 0^{+}}\left\|\boldsymbol{u}-\boldsymbol{u}^{h}\right\|_{V}=0$.
\end{proof}

\subsection{The convergence order of the error estimate}
In the previous subsection, under the hypotheses $H_{1}$ and $H_{2}$, we discussed the convergence of the numerical solution, but did not verify whether the two assumptions hold. Then, in this subsection, the specific finite-dimensional space $\mathcal{K}^{h}$ will be defined. On this basis, the hypotheses $H_{1}$ and $H_{2}$ will be verified, and the order of convergence will be further given. It is worth noting that there are many ways to define $\mathcal{K}^{h}$, and the $\mathcal{K}^{h}$ given here is the most commonly used finite element space in the finite element method.

It is known that the multilayer elastic system is composed of $n$ layers of elastic bodies, and a nondegenerate quasiuniform subdivision of the $i$-th layer of elastic body $\Omega^{i}$ is denoted as $\mathcal{E}^{i}_{h}=\{E^{i}_{1},E^{i}_{2}, \ldots,E^{i}_{N^{i}_{h}}\}$\cite{riviere2003discontinuous,ciarlet2002finite}, where $E^{i}_{j}$ ($i=1,2,\ldots,n$, $j=1,2,\ldots,N^{i}_{h}$) stands for regular finite element and $h$ stands for the maximum diameter of all elements. Then let the finite element partition of the total elastic body region $\Omega$ be $\mathcal{E}_{h}=\cup_{i=1}^{n}\mathcal{E}^{i}_{h}$. Let $V^{ih}\subset V^{i}$ be the finite 
dimension space consisting of piecewise linear polynomial functions corresponding to subdivision $\mathcal{E}^{i}_{h}$ and $V^{h}=V^{1h}\times V^{2h}\times \cdots \times V^{nh}$. Therefore, the finite element space can be defined as $\mathcal{K}^{h}=V^{h}\cap\mathcal{K}$.
When the regularity of the solution $\boldsymbol{u}$ is high, the piecewise higher-order polynomial function can be used as a finite element basis function, and then a higher-order finite element space can be constructed.

In order to verify the hypothesis $H_{1}$, it is also necessary to define a dense subset in the $\mathcal{K}$ space, for which the following conclusions need to be introduced:
\begin{prp}[\cite{adams2003sobolev}]
Let $U \subset \mathbb{R}^{d}, d \geqslant 1$, be an open, bounded, Lipschitz domain and $L_{1} \subset \partial U$ be a relatively open set with a Lipschitz relative boundary. Then the space 
$$ 
X=\left\{x \in C^{\infty}(\bar{U}): x=0 \text{ in a neighborhood of } L_{1}\right\}
$$ 
is dense in 
$$
Y = \left\{y \in H^{1}(U): y= 0 \text{ a.e. on } L_{1}\right\}.
$$
\end{prp}
From this property, the dense subset in space $V^{i}$ can be defined as:
$$
V_{0}^{i}=\left\{\boldsymbol{v}^{i} \in \left(C^{\infty}(\bar{\Omega}^{i})\right)^{d}: \boldsymbol{v}^{i}=\mathbf{0} \text { in a neighborhood of } \Gamma_{1}^{i}\right\}.
$$
Let $V_{0}=V_{0}^{1}\times \cdots \times V_{0}^{n}$. Then the dense subset $\mathcal{K}_{0}$ in $\mathcal{K}$ can be defined as:
$$
\mathcal{K}_{0} = \mathcal{K} \cap V_{0}.
$$

Defining $\Pi^{h}: \mathcal{K} \rightarrow \mathcal{K}^{h}$ as a piecewise linear interpolation operator\cite{bergh2012interpolation}, and then combining the definition and properties of operator $\mathcal{P}^{h}$, we can deduce:
$$
\left\|\boldsymbol{v}-\mathcal{P}^{h} \boldsymbol{v}\right\|_{V} \leqslant\left\|\boldsymbol{v}-\Pi^{h} \boldsymbol{v}\right\|_{V} \leqslant c h\|\boldsymbol{v}\|_{H^{2}} \quad \forall \boldsymbol{v} \in \mathcal{K}_{0}.
$$
Hence, Hypothesis $H_{1}$ is valid.

Based on the definition (\ref{2.16}), it is easy to prove that $j(\cdot,\cdot)$ is Lipschitz continuous with respect to the second term. Naturally, Hypothesis $H_{2}$ is also verified.

Therefore, based on Theorem 3.2, it can be known that when $M>m$, there are unique solutions $\boldsymbol{u}$ and $\boldsymbol{u}^{h}$ for problem $P_{1}$ and $P_{1}^{h}$, respectively, and when the finite element space $\mathcal{K}^{h}$ is defined as above, the following convergence relationship holds according to Theorem 4.3:
$$
\lim_{h \to 0^{+}}\left\|\boldsymbol{u}-\boldsymbol{u}^{h}\right\|_{V} = 0 .
$$

Given that the numerical solution is known to converge to the true solution, the next goal is to determine the order of convergence, but this requires regularity conditions for the solution. Suppose $\boldsymbol{\sigma}^{i}(\boldsymbol{u}^{i})$ satisfies the basic smooth condition, that is:
$$
\boldsymbol{\sigma}^{i}(\boldsymbol{u}^{i})\in Q^{i}_{1} ~\text{and}~  \boldsymbol{u}^{i}\in \left(H^{2}(\Omega^{i})\right)^{d}, ~i=1,2,\ldots,n.
$$
Then, based on the original equations (\ref{2.1})-(\ref{2.7}) and (\ref{4.4}), the following derivation holds:
\begin{align*}
R(\boldsymbol{u},\boldsymbol{v}^{h}) =& (A\boldsymbol{u},\boldsymbol{v}^{h} -\boldsymbol{u})_{V} + (\boldsymbol{f},\boldsymbol{u} - \boldsymbol{v}^{h} )_{V} - j(\boldsymbol{u},\boldsymbol{u}) + j(\boldsymbol{u},\boldsymbol{v}^{h})\\
=& \sum_{i=1}^{n}\int_{\Omega^{i}}\boldsymbol{\sigma}^{i}(\boldsymbol{u}^{i}): \left( \boldsymbol{\varepsilon}\left(\boldsymbol{v}^{hi}\right) - \boldsymbol{\varepsilon}\left(\boldsymbol{u}^{i}\right) \right) dx - (\boldsymbol{f},\boldsymbol{v}^{h} - \boldsymbol{u} )_{V} \\ 
& - j(\boldsymbol{u},\boldsymbol{u}) + j(\boldsymbol{u},\boldsymbol{v}^{h})\\
=& \sum_{i=1}^{n}\left( \int_{\Gamma^{i}} \boldsymbol{\sigma}^{i}(\boldsymbol{u}^{i})\cdot \boldsymbol{\nu} \cdot\left(\boldsymbol{v}^{hi} -\boldsymbol{u}^{i}\right)d a \right) - \int_{\Gamma^{1}_{2}} \boldsymbol{f}_{2}\cdot \left( \boldsymbol{v}^{h1}-\boldsymbol{u}^{1} \right) d a\\ 
&- j(\boldsymbol{u},\boldsymbol{u}) + j(\boldsymbol{u},\boldsymbol{v}^{h})\\
=& \sum_{i=1}^{n-1}\left( \int_{\Gamma^{i}_{c}} \boldsymbol{\sigma}^{i}\cdot \boldsymbol{\beta}^{i} \cdot\left(\boldsymbol{v}^{hi} -\boldsymbol{u}^{i}\right) - \boldsymbol{\sigma}^{i+1}\cdot \boldsymbol{\alpha}^{i+1} \cdot\left(\boldsymbol{v}^{h(i+1)} -\boldsymbol{u}^{i+1}\right)d a \right)\\ 
& +\int_{\Gamma_{3}^{n}} {\sigma}^{n}_{\beta} \cdot  \left({v}^{n}_{\beta} - {u}^{n}_{\beta}\right) + \boldsymbol{\sigma}^{n}_{\eta} \cdot \left(\boldsymbol{v}^{n}_{\eta} - \boldsymbol{u}^{n}_{\eta} \right) d a + j(\boldsymbol{u},\boldsymbol{v}^{h})- j(\boldsymbol{u},\boldsymbol{u})\\
=& \sum_{i=1}^{n-1}\left( \int_{\Gamma^{i}_{c}} {\sigma}^{i}_{N} \left([{v}^{hi}_{N}] - [{u}^{i}_{N}]\right) + \boldsymbol{\sigma}^{i}_{T}  \cdot\left([\boldsymbol{v}^{hi}_{T}] - [\boldsymbol{u}^{i}_{T}] \right)d a \right)\\ 
& +\int_{\Gamma_{3}^{n}} {\sigma}^{n}_{\beta} \cdot  \left({v}^{n}_{\beta} - {u}^{n}_{\beta}\right) + \boldsymbol{\sigma}^{n}_{\eta} \cdot \left(\boldsymbol{v}^{n}_{\eta} - \boldsymbol{u}^{n}_{\eta} \right) d a + j(\boldsymbol{u},\boldsymbol{v}^{h})- j(\boldsymbol{u},\boldsymbol{u})
\end{align*}
Based on the contact conditions of the multilayer elastic system, it can be known that: ${\sigma}^{i}_{N}[{u}^{i}_{T}]=0$ and ${\sigma}^{i}_{N}[{v}^{hi}_{N}]\geqslant 0$. Therefore, according to the definition of functional $j(\cdot,\cdot)$, the following equation holds:
\begin{align*}
 R(\boldsymbol{u},\boldsymbol{v}^{h})
=& \sum_{i=1}^{n-1} \int_{\Gamma^{i}_{c}} {\sigma}^{i}_{N} [{v}^{hi}_{N}] + \boldsymbol{\sigma}^{i}_{T}  \cdot\left([\boldsymbol{v}^{hi}_{T}] - [\boldsymbol{u}^{i}_{T}] \right)d a \\
& +\int_{\Gamma_{3}^{n}} {\sigma}^{n}_{\beta} \cdot  \left({v}^{nh}_{\beta} - {u}^{n}_{\beta}\right)  + \boldsymbol{\sigma}^{n}_{\eta} \cdot \left(\boldsymbol{v}^{nh}_{\eta} - \boldsymbol{u}^{n}_{\eta} \right) d a \\
&+ \sum_{i=1}^{n-1}\int_{\Gamma_{c}^{i}} g_{T}^{i}\left(\sigma_{N}^{i}(\boldsymbol{u}^{i})\right) \left( |[\boldsymbol{v}_{T}^{hi}]| - |[\boldsymbol{u}_{T}^{i}]| \right) d a \\
& + \int_{\Gamma_{3}^{n}} g_{N}^{n}\left(u_{\beta}^{n}\right) \left( v_{\beta}^{nh} - u_{\beta}^{n} \right) + g_{T}^{n}\left(u_{\beta}^{n}\right)  \left( |\boldsymbol{v}_{\eta}^{nh}| - |\boldsymbol{u}_{\eta}^{n}| \right) d a.
\end{align*}
It can be found from formulations (\ref{2.7}) that the direction of $[\boldsymbol{u}_{T}^{i}]$ is always opposite to the direction of $\boldsymbol{\sigma}^{i}_{T}$. 
In addition, when $\boldsymbol{v}^{h}$ approaches $\boldsymbol{u}$, the direction of $\boldsymbol{v}^{h}$ is always fully consistent with $\boldsymbol{u}$, and Lemma 4.2 guarantees that when $h\to 0$, $\boldsymbol{v}^{h}$ is sufficiently approach to $\boldsymbol{u}$. Therefore, when $h$ is sufficiently small, the following inequalities always hold:
\begin{align*}
&\inf_{\boldsymbol{v}^{h}\in \mathcal{K}^{h}} |R(\boldsymbol{u},\boldsymbol{v}^{h})|\\
\leqslant& \inf_{\boldsymbol{v}^{h}\in \mathcal{K}^{h}} c\cdot \Big| \sum_{i=1}^{n-1} \int_{\Gamma^{i}_{c}} {\sigma}^{i}_{N} [{v}^{hi}_{N}] +  \left( g_{T}^{i}\left(\sigma_{N}^{i}(\boldsymbol{u}^{i})\right) - |\boldsymbol{\sigma}^{i}_{T}| \right) \left( |[\boldsymbol{v}_{T}^{hi}]| - |[\boldsymbol{u}_{T}^{i}]| \right) d a \\
& +\int_{\Gamma_{3}^{n}} \left(g_{N}^{n}\left(u_{\beta}^{n}\right) - |{\sigma}^{n}_{\beta}| \right)  \left({v}^{nh}_{\beta} - {u}^{n}_{\beta}\right)  + \left( g_{T}^{n}\left(u_{\beta}^{n}\right) - |\boldsymbol{\sigma}^{n}_{\eta}| \right) \left(|\boldsymbol{v}_{\eta}^{nh}| - |\boldsymbol{u}_{\eta}^{n}| \right) d a \Big|\\
\leqslant & \inf_{\boldsymbol{v}^{h}\in \mathcal{K}^{h}} c\cdot \|\boldsymbol{v}^{h} - \boldsymbol{u}\|_{L^{2}(\Gamma_{2}\cup \Gamma_{3})^{d}}.
\end{align*}
Then, the following estimation formula can be obtained in combination with Theorem 4.1:
\begin{align}\label{4.6}
\|\boldsymbol{u} -\boldsymbol{u}^{h}\|_{V} \leqslant c \inf_{\boldsymbol{v}^{h}\in \mathcal{K}^{h}}\left( \|\boldsymbol{u} -\boldsymbol{v}^{h}\|_{V} + \|\boldsymbol{u} - \boldsymbol{v}^{h}  \|_{L^{2}(\Gamma_{2}\cup \Gamma_{3})^{d}}^{\frac{1}{2}} \right).
\end{align}
Furthermore, according to the trace theorem in Sobolev space\cite{adams2003sobolev}:
\begin{align*}
\boldsymbol{u}^{i}\big|_{\Gamma_{2}^{i}\cup \Gamma_{3}^{i}} \in \left(H^{3/2}(\Gamma_{2}^{i}\cup \Gamma_{3}^{i})\right)^{d},~i=1,2\ldots n.
\end{align*}
The piecewise Lagrange interpolation operator on $\Omega^{i}$, $\Gamma^{i}_{2}$ and $\Gamma^{i}_{3}$ ($i=1,2,\ldots,n$) is denoted as $\Xi^{h}$\cite{ciarlet2002finite,bergh2012interpolation}. From the properties of this operator, $\Xi^{h}\boldsymbol{u}\in \mathcal{K}^{h}$ and the following inequality holds:
\begin{align*}
\|\boldsymbol{u}-\Xi^{h}\boldsymbol{u}\|_{V}&\leqslant ch \|\boldsymbol{u}\|_{H^{2}(\Omega)^{d}},\\
\|\boldsymbol{u}-\Xi^{h}\boldsymbol{u}\|_{L^{2}(\Gamma_{2}\cup \Gamma_{3})^{d}} &\leqslant ch^{3/2} \|\boldsymbol{u}\|_{H^{3/2}(\Gamma_{2}^{i}\cup \Gamma_{3}^{i})^{d}},
\end{align*}
Let $\boldsymbol{v}^{h}=\Xi^{h}\boldsymbol{u}$. When $h$ is sufficiently small, the following convergence relation can be deduced from the error estimation formula (\ref{4.6}):
$$
\left\|\boldsymbol{u}-\boldsymbol{u}^{h}\right\|_{V} \leqslant O\left(h^{3 / 4}\right).
$$
It is worth noting that when the regularity of the solution $\boldsymbol{u}$ is higher, the convergence order of the numerical solution is higher. In addition, if the properties of operator $\mathcal{A}^{i}$ and function $g_{j}^{i}$ are more explicit, the convergence properties here will be more precise.
\begin{rem}
The approximate analysis of the numerical solution shows that an accurate numerical solution can be obtained by using the finite element method for the variational inequality. And because in the variational inequality, the boundary conditions of the original equation system are greatly simplified and only the element space needs to be constrained, the complexity and computational efficiency of the finite element algorithm are also improved. Furthermore, it is worth noting that since the object to be solved is variational inequality, the finite element algorithm needs to be combined with the optimization algorithm, which will be beneficial to optimize the program using optimization theory. In conclusion, the analysis of pavement mechanical response model under the framework of variational inequality has its advantages over the traditional finite element method.
\end{rem}

\section{Conclusion}
So far, the existence and uniqueness of the solution of the multilayer elastic system under the framework of variational inequalities and the approximate error analysis based on the finite element numerical solution have been verified. The above results not only ensure the feasibility of studying the pavement mechanical response model under the finite element framework, but also show many advantages over the traditional model, such as the equivalence of the problem, the simplification of boundary conditions, and the flexibility and effectiveness of the design algorithm.

We believe that the follow-up research work can be carried out from the following three aspects: extending the elastic constitutive equation to the viscoelastic constitutive equation, which will more realistically simulate the asphalt pavement conditions; extending the friction boundary condition to the contact boundary condition with a bonding term, which will better model the complex contact state between layers; using rich optimization theory to write efficient finite element programs of variational inequalities, and comparing numerical solutions with actual data to improve the model.

\section*{Acknowledge}
This work was supported by the National Key Research and Development Project of China under Grant 2020YFA0714301 and the National Natural Science Foundation of China under Grant 61833005.

\appendix
\section{Appendix}
In this section, the specific process of deriving its variational form (\ref{2.19}) from the system of partial differential equations (\ref{2.1})-(\ref{2.7}) will be introduced.

Multiplying both sides of the equilibrium equation (\ref{2.2}) for the $i$-layer elastic body by $\boldsymbol{v}^{i}-\boldsymbol{u}^{i}$ and integrating in the region $\Omega^{i}$ yields:
\begin{equation}\label{a.1}
\begin{aligned}
&\int_{\Gamma_{2}^{i}} \boldsymbol{\sigma}^{i} \cdot \alpha^{i} \cdot \left(\boldsymbol{v}^{i} - \boldsymbol{u}^{i}\right) d a + \int_{\Gamma_{3}^{i}} \boldsymbol{\sigma}^{i} \cdot \beta^{i} \cdot \left(\boldsymbol{v}^{i} - \boldsymbol{u}^{i}\right) d a \\
& + \int_{\Omega^{i}} \boldsymbol{f}^{i}_{0}\left(\boldsymbol{v}^{i} - \boldsymbol{u}^{i}\right) d x 
= \int_{\Omega^{i}} \boldsymbol{\sigma}^{i} :  \left(\boldsymbol{\varepsilon}(\boldsymbol{v}^{i}) - \boldsymbol{\varepsilon}(\boldsymbol{u}^{i})\right) d x
\end{aligned}
\end{equation}
In this process, Green's formula in tensor form holds:
$$
(\boldsymbol{\sigma}^{i}, \boldsymbol{\varepsilon}(\boldsymbol{v}^{i}))_{L^{2}\left(\Omega^{i}\right)^{d}} + (\operatorname{Div} \boldsymbol{\sigma}, \boldsymbol{v}^{i})_{L^{2}\left(\Omega^{i}\right)} = \int_{\Gamma^{i}} \boldsymbol{\sigma}^{i} \cdot \nu \cdot \boldsymbol{v}^{i} \mathrm{d} a ~~~ \forall \boldsymbol{v} \in V^{i},
$$
where $\nu$ represents unit outer normal on $\Gamma^{i}$.
Then, by accumulating the formula (\ref{a.1}), it can be obtained:
\begin{align}\label{a.2}
& a(\boldsymbol{u},\boldsymbol{v}-\boldsymbol{u}) \nonumber \\
=& \sum_{i=1}^{n} \int_{\Gamma_{2}^{i}} \boldsymbol{\sigma}^{i} \cdot \alpha^{i} \cdot \left(\boldsymbol{v}^{i} - \boldsymbol{u}^{i}\right) d a + \sum_{i=1}^{n}\int_{\Gamma_{3}^{i}} \boldsymbol{\sigma}^{i} \cdot \beta^{i} \cdot \left(\boldsymbol{v}^{i} - \boldsymbol{u}^{i}\right) d a \nonumber \\ &+ \sum_{i=1}^{n}\int_{\Omega^{i}} \boldsymbol{f}^{i}_{0}\left(\boldsymbol{v}^{i} - \boldsymbol{u}^{i}\right) d x \nonumber \\
=& \sum_{i=1}^{n-1}\left( \int_{\Gamma_{2}^{i+1}} \boldsymbol{\sigma}^{i+1} \cdot \alpha^{i+1} \cdot \left(\boldsymbol{v}^{i+1} - \boldsymbol{u}^{i+1}\right) d a + \int_{\Gamma_{3}^{i}} \boldsymbol{\sigma}^{i} \cdot \beta^{i} \cdot \left(\boldsymbol{v}^{i} - \boldsymbol{u}^{i}\right) d a \right) \nonumber \\ 
& + L(\boldsymbol{v}-\boldsymbol{u}) + \int_{\Gamma_{3}^{n}} \boldsymbol{\sigma}^{n} \cdot \beta^{n} \cdot \left(\boldsymbol{v}^{n} - \boldsymbol{u}^{n}\right) d a.
\end{align}
In formula \ref{a.2}, based on the decomposition of normal and tangential forces and displacements, the third term satisfies:
\begin{align}\label{a.3}
&\int_{\Gamma_{3}^{n}} \boldsymbol{\sigma}^{n} \cdot \beta^{n} \cdot \left(\boldsymbol{v}^{n} - \boldsymbol{u}^{n}\right) d a \nonumber\\ =&\int_{\Gamma_{3}^{n}} {\sigma}^{n}_{\beta} \cdot  \left({v}^{n}_{\beta} - {u}^{n}_{\beta}\right) d a + \int_{\Gamma_{3}^{n}} \boldsymbol{\sigma}^{n}_{\eta} \cdot \left(\boldsymbol{v}^{n}_{\eta} - \boldsymbol{u}^{n}_{\eta} \right) d a \nonumber\\
=&\int_{\Gamma_{3}^{n}} -g^{n}_{N}(u_{\beta}^{n}) \cdot  \left({v}^{n}_{\beta} - {u}^{n}_{\beta}\right) d a + \int_{\Gamma_{3}^{n}} \boldsymbol{\sigma}^{n}_{\eta} \cdot \left(\boldsymbol{v}^{n}_{\eta} - \boldsymbol{u}^{n}_{\eta} \right) d a
\end{align}
Similarly, based on the decomposition of force and displacement in normal and tangential directions and equation (\ref{2.6}), the first term of equation (\ref{a.2}) satisfies:
\begin{equation}\label{a.4}
\begin{aligned}
&\sum_{i=1}^{n-1}\left( \int_{\Gamma_{2}^{i+1}} \boldsymbol{\sigma}^{i+1} \cdot \alpha^{i+1} \cdot \left(\boldsymbol{v}^{i+1} - \boldsymbol{u}^{i+1}\right) d a + \int_{\Gamma_{3}^{i}} \boldsymbol{\sigma}^{i} \cdot \beta^{i} \cdot \left(\boldsymbol{v}^{i} - \boldsymbol{u}^{i}\right) d a \right)\\
=&\sum_{i=1}^{n-1}\Bigg(\int_{\Gamma_{2}^{i+1}} {\sigma}^{i+1}_{\alpha} \cdot \left({v}^{i+1}_{\alpha} - {u}^{i+1}_{\alpha}\right) d a + \int_{\Gamma_{2}^{i+1}} \boldsymbol{\sigma}^{i+1}_{\tau} \cdot \left(\boldsymbol{v}^{i+1}_{\tau} - \boldsymbol{u}^{i+1}_{\tau}\right) d a \\
&+ \int_{\Gamma_{3}^{i}} {\sigma}^{i}_{\beta} \cdot \left({v}^{i}_{\beta} - {u}^{i}_{\beta}\right) d a + \int_{\Gamma_{3}^{i}} \boldsymbol{\sigma}^{i}_{\eta} \cdot \left(\boldsymbol{v}^{i}_{\eta} - \boldsymbol{u}^{i}_{\eta}\right) da \Bigg)\\
=& \sum_{i=1}^{n-1}\left(\int_{\Gamma_{c}^{i}} {\sigma}^{i}_{N} \cdot \left(\left[v_{N}^{i}\right]-\left[u_{N}^{i}\right]\right) d a + \int_{\Gamma_{c}^{i}} \boldsymbol{\sigma}^{i}_{T} \cdot \left(\left[\boldsymbol{v}_{T}^{i}\right] - \left[\boldsymbol{u}_{T}^{i}\right]\right) d a\right)
\end{aligned}
\end{equation}
According to the mutual non-penetration condition, ${\sigma}^{i}_{N}\cdot [u_{N}^{i}] = 0$ and when $\boldsymbol{v}\in\mathcal{K}$, ${\sigma}^{i}_{N}\cdot [v_{N}^{i}] \geqslant 0$.
Therefore, by substituting formulations (\ref{a.3}) - (\ref{a.4}) into equation (\ref{a.2}) and combining with boundary conditions (\ref{2.5}) and (\ref{2.8}), the following inequality relation can be derived:
\begin{align*}
& a(\boldsymbol{u},\boldsymbol{v}-\boldsymbol{u})\\
=&  L(\boldsymbol{v}-\boldsymbol{u}) - \int_{\Gamma_{3}^{n}} g^{n}_{N}(u_{\beta}^{n}) \cdot \left({v}^{n}_{\beta} - {u}^{n}_{\beta}\right) d a + \int_{\Gamma_{3}^{n}} \boldsymbol{\sigma}^{n}_{\eta} \cdot \left(\boldsymbol{v}^{n}_{\eta} - \boldsymbol{u}^{n}_{\eta} \right) d a\\
& + \sum_{i=1}^{n-1} \int_{\Gamma_{c}^{i}} {\sigma}^{i}_{N} \cdot \left(\left[v_{N}^{i}\right]-\left[u_{N}^{i}\right]\right) d a + \sum_{i=1}^{n-1} \int_{\Gamma_{c}^{i}} \boldsymbol{\sigma}^{i}_{T} \cdot \left(\left[\boldsymbol{v}_{T}^{i}\right] - \left[\boldsymbol{u}_{T}^{i}\right]\right) d a\\
\geqslant & L(\boldsymbol{v}-\boldsymbol{u}) - \int_{\Gamma_{3}^{n}} g^{n}_{N}(u_{\beta}^{n}) \cdot  \left({v}^{n}_{\beta} - {u}^{n}_{\beta}\right) d a - \int_{\Gamma_{3}^{n}} g^{n}_{T}(u_{\beta}^{n}) \cdot \left(|\boldsymbol{v}^{n}_{\eta}| - |\boldsymbol{u}^{n}_{\eta}| \right) d a\\
& - \sum_{i=1}^{n-1} \int_{\Gamma_{c}^{i}} g^{i}_{T}(\sigma_{N}^{i}(\boldsymbol{u}^{i})) \cdot \left(|[\boldsymbol{v}_{T}^{i}]| - |[\boldsymbol{u}_{T}^{i}]|\right) d a\\
\geqslant & L(\boldsymbol{v}-\boldsymbol{u}) - j(\boldsymbol{u},\boldsymbol{v}) + j(\boldsymbol{u},\boldsymbol{u}).
\end{align*}

Hence, problem $P_{1}$ can be derived from problem $P_{0}$, that is, the solution of problem $P_{0}$ must be the solution of problem $P_{1}$. If the solution of problem $P_{1}$ exists and is unique, then problem $P_{1}$ is equivalent to problem $P_{0}$.

\bibliographystyle{unsrt}
\bibliography{ref}

\end{document}